\documentclass[11pt,reqno, english]{amsart}  
\usepackage[utf8]{inputenc}
\usepackage[T1]{fontenc}
\usepackage{amsmath,amsthm}
\usepackage{amsfonts,amssymb,enumerate}
\usepackage{url}
\usepackage{mathtools}  
\usepackage{enumerate, paralist}
\usepackage[colorlinks,cite color=blue,pagebackref=true,pdftex]{hyperref}
\usepackage{anysize}
\usepackage{units}
\usepackage{tikz}
\usepackage{float}
\usepackage{algpseudocode}

\newtheorem{theorem}{Theorem}[section]
\newtheorem*{theorem*}{Theorem}
\newtheorem*{lemma*}{Lemma}
\newtheorem*{prop*}{Proposition}

\newtheorem{cor}[theorem]{Corollary}
\newtheorem{prop}[theorem]{Proposition}

\newtheorem{quest}[theorem]{Question}

\theoremstyle{definition}
\newtheorem{example}[theorem]{Example}

\newtheorem{defn}[theorem]{Definition}
\newtheorem{remark}[theorem]{Remark}

\renewcommand{\emptyset}{\varnothing}

\newcommand{\R}{\mathbb{R}}
\newcommand{\Z}{\mathbb{Z}}
\newcommand{\conv}{\mathrm{conv}}
\newcommand{\cone}{\mathrm{cone}}

\newcommand\WT{w}%
\newcommand\NO{\eta}%
\newcommand\GI{\mathrm{GI}}%
\newcommand\pSE{\mathrm{pSE}}%
\newcommand\MS{\mathrm{MS}}%

\newcommand\ED{\mathrm{ED}}%
\newcommand\SP{\mathrm{SP}}%

\newcommand\eps{\varepsilon}%

\newcommand\arb{\mathcal{A}}%
\newcommand\PPGKZ{\psi}%
\newcommand\PP{\Pi}%
\newcommand\NP[2][~]{\Gamma_{#2}^{#1}}%
\newcommand\br{\mathcal{B}}%
\newcommand\Mon{\Sigma}%

\newcommand\vopt{v_{opt}}%
\newcommand\vmopt{v_{-opt}}%

\newcommand\Def[1]{\textbf{#1}}%

\newcommand\Nb[2]{\mathrm{Nb}_{#1}(#2)}%

\DeclareMathOperator{\argmax}{argmax}

\newcommand\NONCOH[1][1.3]{
\begin{tikzpicture}[scale=#1,>=latex,->,thick,color=gray]
\tikzstyle Y=[color=red];

\draw[Y] (0.5,0.5) -- (1.5,0.5);
\draw (0.5,0.5) -- (0.5,1.5);

\draw[Y] (0.5,1.5) -- (1.5,1.5);

\draw[Y] (1.5,0.5) -- (1.5,1.5);

\draw (0,0) -- (1,0);
\draw[Y] (0,0) -- (0.5,0.5);
\draw (0,0) -- (0,1);

\draw[color=white,ultra thick] (0,1) -- (1,1);
\draw[Y] (0,1) -- (1,1);
\draw (0,1) -- (0.5,1.5);

\draw[color=white,ultra thick] (1,0) -- (1,1);
\draw[Y] (1,0) -- (1,1);
\draw (1,0) -- (1.5,0.5);

\draw[Y] (1,1) -- (1.5,1.5);
\end{tikzpicture}
}

\newcommand\COH[1][1.3]{
\begin{tikzpicture}[scale=#1,>=latex,->,thick,color=gray]
\tikzstyle Y=[color=red,thick];

\draw[Y] (0.5,0.5) -- (1.5,0.5);
\draw (0.5,0.5) -- (0.5,1.5);

\draw[Y] (0.5,1.5) -- (1.5,1.5);

\draw[Y] (1.5,0.5) -- (1.5,1.5);

\draw[Y] (0,0) -- (1,0);
\draw (0,0) -- (0.5,0.5);
\draw (0,0) -- (0,1);

\draw[color=white,ultra thick] (0,1) -- (1,1);
\draw[Y] (0,1) -- (1,1);
\draw (0,1) -- (0.5,1.5);

\draw[color=white,ultra thick] (1,0) -- (1,1);
\draw[Y] (1,0) -- (1,1);
\draw (1,0) -- (1.5,0.5);

\draw[Y] (1,1) -- (1.5,1.5);
\end{tikzpicture}
}

\newcommand\COHH[1][1.3]{
\begin{tikzpicture}[scale=#1,>=latex,->,thick,color=gray]
\tikzstyle Y=[color=red];

\draw[Y] (0.5,0.5) -- (1.5,0.5);
\draw (0.5,0.5) -- (0.5,1.5);

\draw[Y] (0.5,1.5) -- (1.5,1.5);

\draw[Y] (1.5,0.5) -- (1.5,1.5);

\draw (0,0) -- (1,0);
\draw (0,0) -- (0.5,0.5);
\draw[Y] (0,0) -- (0,1);

\draw[color=white,ultra thick] (0,1) -- (1,1);
\draw[Y] (0,1) -- (1,1);
\draw (0,1) -- (0.5,1.5);

\draw[color=white,ultra thick] (1,0) -- (1,1);
\draw[Y] (1,0) -- (1,1);
\draw (1,0) -- (1.5,0.5);

\draw[Y] (1,1) -- (1.5,1.5);
\end{tikzpicture}
}

\newcommand\FP[1][1.3]{
\begin{tikzpicture}[scale=#1,>=latex,->,thick,color=gray]
\tikzstyle Y=[color=blue];

\draw[Y] (0.5,0.5) -- (1.5,0.5);
\draw (0.5,0.5) -- (0.5,1.5);

\draw[Y] (0.5,1.5) -- (1.5,1.5);

\draw[Y] (1.5,0.5) -- (1.5,1.5);

\draw (0,0) -- (1,0);
\draw (0,0) -- (0.5,0.5);
\draw[Y] (0,0) -- (0,1);

\draw[color=white,ultra thick] (0,1) -- (1,1);
\draw[Y] (0,1) -- (1,1);
\draw[Y] (0,1) -- (0.5,1.5);

\draw[color=white,ultra thick] (1,0) -- (1,1);
\draw[Y] (1,0) -- (1,1);
\draw (1,0) -- (1.5,0.5);

\draw[Y] (1,1) -- (1.5,1.5);
\end{tikzpicture}
}

\begin{document}
\newcommand\NW{\mathrm{NW}}%
\renewcommand\int{\mathrm{int}}%
\newcommand\Ncone{\mathcal{N}}%
\newcommand\EZ{\mathcal{E}}%

\title[The Polyhedral Geometry of Pivot Rules and Monotone Paths]{The Polyhedral Geometry of \\ Pivot Rules and Monotone Paths}

\author[A.~Black \and J.~De Loera \and N.~L\"utjeharms \and R.~Sanyal]{Alexander
E. Black \and Jes\'{u}s A. De Loera \and Niklas L\"utjeharms \and Raman Sanyal}

\address[A.~Black, J.~De Loera]{Dept.\ Mathematics, Univ. of California,
Davis, CA 95616, USA}
\email{aeblack@ucdavis.edu, deloera@math.ucdavis.edu} 

\address[N.~L\"utjeharms, R.~Sanyal]{Institut f\"ur Mathematik, Goethe-Universit\"at Frankfurt, Germany} 
\email{luetjeharms@stud.uni-frankfurt.de, sanyal@math.uni-frankfurt.de}

\begin{abstract}
    Motivated by the analysis of the performance of the simplex method we study
    the behavior of families of pivot rules of linear programs. We introduce
    \emph{normalized-weight pivot rules} which are fundamental for the following
    reasons: First, they are \emph{memory-less}, in the sense that the pivots
    are governed by local information encoded by an arborescence. Second, many
    of the most used pivot rules belong to that class, and we show this subclass
    is critical for understanding the complexity of all pivot rules. Finally,
    normalized-weight pivot rules can be parametrized in a natural continuous
    manner.

    We show the existence of two polytopes, the \emph{pivot rule polytopes} and
    the \emph{neighbotopes}, that capture the behavior of normalized-weight
    pivot rules on polytopes and linear programs. We explain their face
    structure in terms of multi-arborescences.  We compute upper bounds on the
    number of coherent arborescences, that is, vertices of our polytopes.

    Beyond optimization, our constructions provide new perspectives on
    classical geometric combinatorics. We introduce a normalized-weight pivot
    rule, we call the \emph{max-slope pivot rule} which generalizes the
    shadow-vertex pivot rule. The corresponding pivot rule polytopes and
    neighbotopes refine \emph{monotone path polytopes} of Billera--Sturmfels.
    Moreover special cases of our polytopes yield permutahedra, associahedra,
    and multiplihedra. For the greatest improvement pivot rules we draw
    connections to sweep polytopes and polymatroids.
\end{abstract}

\maketitle

\section{Introduction}\label{sec:intro}

For $A \in \R^{n \times d}, b \in \R^n, c \in \R^d$ we consider the linear program (LP)
\[
    \begin{array}{ll}
        \max & c^t x \\
        \text{s.t.} & Ax \le b 
    \end{array}
\]
The \emph{simplex method} is one of the most popular algorithms for
solving linear programs (see
\cite{bertsimastsitsiklis,Dantzigbook,Schrijver1986}).  The key
ingredient, which is decisive for the running time on a given
instance, is the choice of a \emph{pivot rule}.  Since the
inception of the simplex algorithm, many different pivot rules have
been proposed and analyzed. Starting with Klee and Minty in 1972 \cite{kleeminty} many of the popular pivot rules have been
shown to require an exponential number of steps; see 
\cite{amentaziegler,AvisFriedmann2017,Friedmannetal,Hansen+Zwick2015,terlaky,zadeh2,zieglerextremelps}
and references there. To this day, no pivot rule is known
to take only polynomially many steps on every LP. In this paper
we study the behavior of parametric families of pivot rules and
uncover a rich polyhedral structure. We define polytopes whose
geometry capture the behavior of pivot rules on given LPs.  This
provides a new perspective on the study of the performance of the
simplex method. 

Our constructions are also of interest to the (geometric)
combinatorics community. A generic linear function $c$ induces an
acyclic orientation on the graph of the polytope $P = \{ x \in \R^d
: Ax \le b \}$. The collection of $c$-monotone paths has a natural
topological structure that is studied under the name \emph{Baues
poset}. In the seminal paper~\cite{BKS}, Billera, Kapranov, and
Sturmfels showed that the Baues poset has the homotopy type of a sphere and is represented by the boundary of the \emph{monotone path polytope}
from~\cite{BSFiberPoly}. The vertices of monotone path polytopes are in
bijection to special monotone paths, called \emph{coherent}.  Later, many
important combinatorial constructions and polytopes arose as
monotone path polytopes.  By replacing $c$-monotone paths by
$c$-monotone arborescences, our constructions provide a
generalization of the theory of monotone path polytopes and
many prominent combinatorial polytopes.

Formalizing the notion of a pivot rule is complicated; for example, several
authors showed that pivot rules can be used to encode problems that are hard
in the sense of complexity theory
\cite{adler+papadimitriou+rubinstein,disser+skutella, fearnley+savani}. We
will not try to give a precise definition of what constitutes a pivot rule
because our  taxonomy of pivot rules relies only on a polyhedral geometry
perspective. Throughout, we will refer to $(P,c)$ as the \Def{linear program}.
Geometrically the simplex method finds a $c$-monotone path in the graph of $P$
from any initial vertex $v$ of $P$ to the optimal vertex $\vopt$. The
algorithm proceeds along directed edges.  At any non-sink $v$, the pivot rule
chooses a neighboring vertex $u$ of $v$ with $c^t u > c^t v$.  

\begin{defn}
    The \Def{footprint} of a pivot rule $R$ on an LP $(P,c)$
    is the directed acyclic subgraph obtained as the union of all
    $c$-monotone paths produced by $R$ for every starting
    vertex.
    The pivot rule $R$ is a \Def{memory-less} pivot rule if
    its footprint for every LP is an \Def{arborescence}, i.e., a
    directed tree with root at the optimal vertex $\vopt$. 
\end{defn}

\begin{figure}[t]
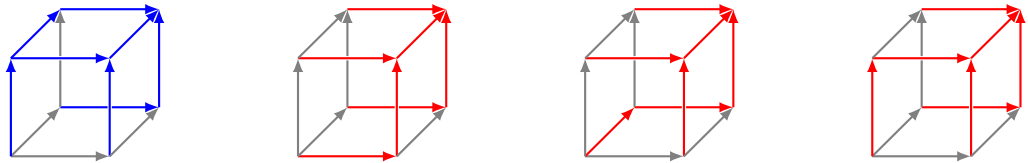

    \centering
    \FP
    \qquad
    \qquad
    \COH
    \qquad
    \qquad
    \NONCOH
    \qquad
    \qquad
    \COHH
    \caption{A footprint and three arborescences on the unit cube. The second
    arborescence from the left does not come from a NW-rule.}
    \label{fig:cube_arbs}
\end{figure}

Figure~\ref{fig:cube_arbs} shows a footprint and three such arborescences on a
$3$-cube.  Equivalently, a pivot rule is memory-less if it chooses
the neighbor of $v \neq \vopt$ using only \emph{local} information
provided by the set of neighbors $\Nb{P}{v}$ of $v$. Many rules
that are used in practice, including \emph{greatest improvement}
and \emph{steepest edge}, are memory-less (c.f.
Section~\ref{sec:LPs}). Pivot rules not in this class include
\emph{Zadeh's least-entered facet} rule as well as the original
\emph{shadow vertex} rule. 

For a given LP $(P,c)$ a memory-less pivot rule is represented by its
arborescence, that is, at every vertex, the choice made by that pivot rule is
encoded in the outgoing arc of the arborescence. In particular, every
memory-less pivot rule corresponds to a choice of an arborescence for every LP
$(P,c)$.   From this perspective, for every pivot-rule, there
is a memory-less pivot rule, given by the shortest-path arborescence of the
footprint, which takes at most the same number of steps. In consequence, if
every memory-less pivot rule takes exponentially many steps, then so does
every pivot rule.

The two main questions that we address in this paper are
\begin{enumerate}[(A)]
    \item How do the arborescences vary for fixed  objective function $c$ and
        varying pivot rule?
    \item How do the arborescences vary for fixed pivot rule and
        varying  objective function $c$?
\end{enumerate}

To be able to change the pivot rules in a controlled and continuous manner, we
restrict to the following setup: For given $P \subset \R^d$ and $c \in \R^d$,
choose a \Def{normalization} $\NO : \R^d \to \R$ and a \Def{weight} $\WT \in
\R^d$. For $v \neq \vopt$, the next vertex on the simplex-path from
$v$ to $\vopt$ is 
\begin{equation}\label{eqn:NW}
    u_* \ = \ \argmax \left\{ \frac{\WT^t(u-v)}{\NO(u-v)} : u \text{ adjacent
    to $v$ and }
    c^tu > c^tv \right\} \, .
\end{equation}
A choice of $\WT$ and $\NO$ for given $(P,c)$ is called a
\Def{normalized-weight} pivot rule, or \Def{NW-rule} for short. If
$R$ is a normalized-weight pivot rule, we sometimes write $
\NO^R(P,c)$ and $\WT^R(P,c)$ to stress the dependence of $\NO$ and
$\WT$ on the LP $(P,c)$. NW-rules are memory-less pivot rules: for a fixed LP $(P,c)$ Equation \eqref{eqn:NW}
determines an arborescence $\arb$, that is, a map on the vertices of $P$ with
$\arb(\vopt) = \vopt$ and $\arb(v) = u_*$
otherwise.

As we explain in the next section, several well-known pivot rules
(greatest-improvement, steepest-edge, etc.) as well as the max-slope pivot
rule, a memory-less generalization of the shadow-vertex rule (see below),
belong to that class. While NW-rules are a strict subclass of memory-less
pivot rules, we show that they are universal in the following sense.

\begin{theorem}\label{thm:NW_universality}
    For every simple polytope $P$ there is a perturbation $P'$,
    combinatorially isomorphic to $P$, such that for any memory-less pivot
    rule there is a NW-rule that produces the same arborescence for $(P',c)$
    for every $c$.
\end{theorem}

If $(P,c)$ is a non-degenerate LP, then $(P',c)$ has the same optimal basis
and we may assume that $(P,c)$ is already sufficiently generic.  Hence,
NW-rules are essentially all we need to study to understand memory-less rules.
Furthermore, by our earlier argument, memory-less rules are essentially all we
need to study to understand all pivot rules. We may put these observations
together in the following Corollary.

\begin{cor}
    If there is a pivot rule for which the simplex method takes polynomially
    many steps on every LP, then there is an NW-rule that takes takes
    polynomially many steps on every LP.
\end{cor}

We can continuously change the pivot rule by varying the weight $\WT$.  We call
an arborescence that arises via~\eqref{eqn:NW} for a fixed weight $\WT$ a
\Def{coherent} arborescence and write $\arb = \arb_{P,c}^\NO(\WT)$. This
terminology underlines the proximity to the theory of coherent monotone
paths~\cite{BKS, BSFiberPoly} (see below).

An answer to question (A) is provided by the following theorem. For a polytope
$Q \subseteq \R^d$ and $w \in \R^d$, we write $Q^\WT = \{ x \in Q : \WT^tx \ge
\WT^ty, y \in Q\}$ to denote the face that maximizes $x \mapsto \WT^t x$.

\begin{theorem}\label{thm:PP}
    Let $(P,c)$ be a linear program and $\NO$ a normalization. There
    is a polytope \mbox{$\PP_{P,c}^\NO \subset \R^d$}, called the \Def{pivot rule polytope}
    of $(P,c)$ and $\NO$,
    such that the following holds: For any generic weights $\WT,\WT'$ 
    \[
        \arb_{P,c}^\NO(\WT) \ = \ \arb_{P,c}^\NO(\WT') 
        \quad \Longleftrightarrow \quad
        (\PP_{P,c}^\NO)^\WT \ = \ (\PP_{P,c}^\NO)^{\WT'} \, .
    \]
\end{theorem}

Question (B) is strongly related to parametric linear programming. Whereas a
basic question there is roughly which objective functions yield the same optimum,
we will address the more subtle question which objective functions yield the
same arborescence. We make two assumptions on the NW-rule $R$, namely that
$\NO^R(P,c)$ is independent of $c$ and that $\WT^R(P,c) = c$. Thus, for a
fixed normalization function $\NO$, we will write $\br_{P}^\NO(c) :=
\arb_{P,c}^\NO(c)$ for the arborescence of $(P,c)$ obtained
from~\eqref{eqn:NW} with respect to $\NO$ and weight $\WT=c$. We show 
that the collection of arborescences $\br_{P}^\NO(c)$ is governed by 
another polytope.

\begin{theorem}\label{thm:NP}
    Let $P \subset \R^d$ be a polytope and $\NO$ a normalization.
    There is a polytope $\NP[\NO]{P} \subset \R^d$, called the \Def{neighbotope}
    of $P$ and $\NO$, such that the following holds: For any generic objective
    functions $c,c' \in \R^d$
    \[
        \br_{P}^\NO(c) \ = \ \br_{P}^\NO(c') 
        \quad \Longleftrightarrow \quad
        (\NP[\NO]{P})^c \ = \ (\NP[\NO]{P})^{c'} \, .
    \]
\end{theorem}

We prove Theorems \ref{thm:PP} and \ref{thm:NP} in Section \ref{sec:PP+NP}. 

We describe the face structure of pivot rule polytopes and neighbotopes in
terms of \emph{multi-arborescences} and discuss the relation to general
arborescences of LPs that were studied and enumerated by Athanasiadis et al.\
in~\cite{enumerative}.  In particular, we give bounds on the number of
coherent arborescences; see Section \ref{sec:facesPP+NP}.

As a memory-less version of the shadow-vertex pivot rule
we introduce the \Def{max-slope} (\Def{MS}) pivot rule:
For given LP $(P,c)$ choose $\NO^{\MS}(u-v) = c^t(u-v)$ and $\WT^{\MS} \in \R^d$ generic and linearly independent of $c$. Thus the resulting arborescence $\arb$ satisfies
\begin{equation}\label{eqn:MSintro}
    \arb(v) \ = \ \argmax \left\{ \frac{\WT^t(u-v)}{c^t(u-v)} : u \text{ adjacent
    to $v$ and }
    c^tu > c^tv \right\} \, ,
\end{equation}
for $v \not= \vopt$.

Let $r=P^\WT$ be the vertex selected by $\WT$, the unique path in the
arborescence $\arb$ above, starting at $r$, is precisely the path followed by
the shadow-vertex pivot rule (see Proposition~\ref{prop:shadow_path}). Let
$\vmopt$ be the vertex of $P$ minimizing $c$. The unique path in the
arborescence $\arb$ starting at $\vmopt$ passes through $r$. It is the
\emph{coherent monotone path} of $(P,c)$ with respect to $\WT$ in the sense of
\cite{BSFiberPoly}. For varying $\WT$ the resulting coherent monotone paths are
parametrized by the vertices of the \Def{monotone path polytope} $\Mon_c (P)$.
Obviously the arborescence contains more information than just the monotone path
from $\vmopt$. This refinement can be seen geometrically in terms of Minkowski
sums (see Section \ref{sec:max_slope}).

\begin{theorem} \label{thm:MS_Mon}
    Let $P \subset \R^d$ be a polytope and $c$ a generic objective function.
    Then the monotone path polytope $\Mon_c(P)$ is a 
    weak Minkowski summand of the max-slope pivot rule polytope $\PP_{P,c}^\MS$.
    If $P$ is a zonotope, then $\Mon_c(P)$ is normally equivalent to $\PP_{P,c}^\MS$.
\end{theorem}

Interestingly the construction of pivot rule polytopes is fundamentally
different from that of monotone path polytopes in \cite{BSFiberPoly}. In
particular, the result gives a new way of studying monotone path polytopes of
zonotopes. In Section \ref{sec:examples}  we highlight that Stasheff's
associahedra and multiplihedra can be realized as max-slope pivot rule
polytopes.

The pivot rule polytopes for the greatest improvement pivot rule relate to yet
another important construction from geometric combinatorics going back to
classical work of Goodman and Pollack; see~\cite{GP} and references therein.
The \Def{sweep polytope} $\SP(p_1,\dots,p_n)$, introduced by Padrol and
Philippe in~\cite{Sweep}, captures the orderings of a point configuration
$p_1,\dots,p_n$ induced by varying linear functions. For a polytope $P \subset
\R^d$ and a normalization $\NO$, define the set of \Def{normalized edge
directions} $\ED^\NO(P):= \{ \frac{u-v}{\NO(u-v)} : uv \in E(P) \}$. If $c$ is
a generic objective function, then let $\ED^\NO(P,c):= \{ \frac{u-v}{\NO(u-v)}
: uv \in E(P), c^tu > c^tv \}$ the collection of normalized
\Def{$c$-improving} edge directions.

\begin{theorem}\label{thm:NP_sweep}
    Let $(P,c)$ be a linear program and $\NO$ a normalization.  Then the pivot
    rule polytope $\PP_{P,c}^\NO$ is a weak Minkowski summand of the sweep
    polytope of normalized $c$-improving edge directions $\SP(\ED^\NO(P,c))$.

    Furthermore, the neighbotope $\NP[\NO]{P}$ is a weak Minkowski summand of
    the sweep polytope of normalized edge directions $\SP(\ED^\NO(P))$.
\end{theorem}

\newcommand\RoSy{\Phi}%
\newcommand\Pos{\RoSy^+}%
\newcommand\Simp{\Delta}%
We show that in a particularly interesting case the neighbotope and the sweep
polytope of edge directions are normally equivalent. Let us write $\ED(P) =
\ED^1(P)$ for the unnormalized edge directions.  If $\RoSy \subset \R^n$ is an
irreducible crystallographic root system, then we associate to it the
\Def{Coxeter zonotope} $Z_\RoSy = \frac{1}{2} \sum_{\alpha \in \RoSy}
[-\alpha,\alpha]$. It is easy to see that $\ED(Z_\RoSy) = \RoSy$.

\begin{theorem}\label{thm:Coxeter-Zonotope}
    Let $\RoSy$ be an irreducible crystallographic root system with Coxeter
    zonotope $Z_\RoSy$. Then the greatest-improvement neighbotope
    $\NP[\GI]{Z_\RoSy}$ is normally equivalent to $\SP(\ED(Z_\RoSy)) =
    \SP(\RoSy)$.
\end{theorem}

The proof relies on a result (Theorem~\ref{thm:comp_roots}) on irreducible
crystallographic root systems that is of independent interest: for every pair
$\alpha, \beta$ of elements that are incomparable in the root poset of $\Phi$
there is a simple system $\Delta \subseteq \Phi$ whose only positive roots are
$\alpha$ and $\beta$.

We give several of examples of pivot rule polytopes and neighbotopes in
Section \ref{sec:examples} but defer a detailed discussion to the forthcoming
paper~\cite{PivPoly2}.

From Equation \eqref{eqn:NW} one can see that arborescences for NW-rules in
general are obtained by local greedy choices. For the greatest improvement
pivot rule we show in Section \ref{sec:genie} that its arborescences can be
derived from a basic combinatorial optimization problem, we named the
\textsc{Max Potential Energy Branching}. This problem has the structure of a
polymatroid and can be solved by the greedy algorithm. We explain the
associated polytope in detail, which also justifies the name ``neighbotope''.

\textbf{Acknowledgements.} The first and second author were supported by the
NSF through the NSF graduate research fellowship program and grant
DMS-1818969. They are also grateful for the wonderful hospitality of the
Goethe-Universit\"at Frankfurt where this paper was written. We are grateful
for comments and suggestions from Samuel Fiorini, Martin Skutella, Laura
Sanit\`a, and Christian Stump.

\section{LPs, pivot rules, and arborescences}\label{sec:LPs}

Let $P \subseteq \R^d$ be a fixed polytope. We will denote by $V(P)$ the
\Def{vertex set} of $P$ and by \mbox{$G(P) = (V(P),E(P))$} the \Def{graph} of $P$. A
linear function $c \in \R^d$ is \Def{(edge) generic} if $c^tu \neq c^tv$ for
all edges $uv \in E(P)$. Every generic linear function $c$ induces an acyclic
orientation on $G(P)$ by orienting $v \to u$ if $c^tu > c^tv$. The directed
graph has a unique sink $\vopt$ and, in fact, every subgraph of a face will have a
unique sink. Such an orientation is called a \Def{unique sink orientation} and
we call $(P,c)$ a \Def{linear program}. For a vertex $v \in V(P)$, we write
$\Nb{P}{v} := \{ u : uv \in E(P) \}$ for the \Def{neighbors} of $v$ in $G(P)$
and we write $\Nb{P,c}{v} := \{ u \in \Nb{P}{v} : c^tu > c^tv \}$ for the
\Def{$c$-improving} neighbors.

A \Def{$c$-arborescence} of $P$ is a map $\arb : V(P) \to V(P)$ satisfying
$\arb(v) = v$ if and only if $v = \vopt$ and $\arb(v) \in \Nb{P,c}{v}$ for all
$v \in V(P) \setminus \vopt$. For a memory-less pivot rule, the choice of the
neighboring vertex $u_* \in \Nb{P,c}{v}$ for $v \neq \vopt$ results in a
$c$-arborescence $\arb$, which captures the behavior of the pivot rule on the
linear program $(P,c)$. Arborescences of polytopes have appeared as oracles that allow geometric enumeration output-sensitive algorithms \cite{AVISFUKUDA-1996}. 

For a given normalization $\NO$ and weight $\WT$,~\eqref{eqn:NW} determines 
an arborescence $\arb = \arb_{P,c}^\NO(\WT)$ given by
\begin{equation}\label{eqn:NWarb}
    \arb(v) \ := \ 
    \argmax \left\{ \frac{\WT^t(u-v)}{\NO(u-v)} : u \in \Nb{P,c}{v} 
    \right\} 
\end{equation}
for $v \neq \vopt$ and $\arb(\vopt) := \vopt$.

The following well-known and important pivot rules belong to the class of
NW-rules (this requires the assumption that $P$ is a simple polyhedron):
\begin{itemize}[]
    \item \Def{Greatest improvement} (\Def{GI}): choose $\WT^{\GI} = c$ and
        $\NO^{GI}(u-v) \equiv 1$;
    \item \Def{$p$-Steepest edge} (\Def{pSE}): choose $\WT^{\pSE} = c$ and
        $\NO^{\pSE}(u-v) = \|u-v\|_p$ for some fixed $p \ge 1$;
    \item \Def{Max-slope} (\Def{MS}): choose $\NO^{\MS}(u-v) = c^t(u-v)$ and
        $\WT^{\MS}$ linearly independent of $c$.
\end{itemize}
The max-slope rule is a memory-less version of the shadow-vertex rule, that we
will treat in depth in Section~\ref{sec:max_slope}. Figure~\ref{fig:tetra} shows the
six arborescences of the tetrahedron including the five arborescences obtained
from the max-slope rule. 
\begin{figure}[H]
    \centering
    \includegraphics[width=8cm]{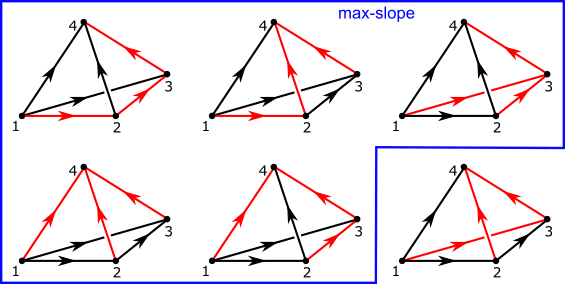}
    \caption{The arborescences of the tetrahedron.}
    \label{fig:tetra}
\end{figure}
It turns out that all $3!=6$ arborescences of the tetrahedron can be obtained
from a NW-rule for a suitable choice of a normalization. However, this is not
true in general. 

Indeed, the arborescence in middle of Figure~\ref{fig:cube_arbs} cannot be obtained from an NW-rule. Observe that any NW-rule makes a choice based only on the set of edge directions 
\[
    D_{P,c}(v) \ := \ \{ u-v : u \in \Nb{P,c}{v} \} \, .
\]
Hence, if two vertices $v,v' \in V(P)$ satisfy $D_{P,c}(v') \subseteq
D_{P,c}(v)$ and $\arb(v)-v \in D_{P,c}(v')$, then $\arb(v)-v = \arb(v') - v'$.
So, the choice of the improving neighbor for $v$ forces the improving neighbor
for $v'$ to be the same. The middle arborescence of Figure~\ref{fig:cube_arbs} violates this constraint.

We call a polytope $P$ \Def{edge-generic} if $u-v \neq u'-v'$ for any two
distinct edges $uv,u'v' \in E(P)$.

\begin{prop}\label{prop:edge_gen}
    Let $P$ be an edge-generic polytope and $c$ a generic objective function.
    For any $c$-arborescence $\arb$ there is a normalization $\NO$ and a weight
    $\WT$ such that $\arb = \arb_{P,c}^\NO(\WT)$.
\end{prop}
\begin{proof}
    It follows from edge-genericity that $D_{P,c}(v) \cap D_{P,c}(v') =
    \emptyset$ for all $v \neq v'$. Define the normalization $\NO : \R^d \to
    \R$ by $\NO(\arb(v)-v) := 1$ for all $v \in V(P)$ and $\NO(x) := \kappa$
    for $x \not\in \{ \arb(v) - v : v \in V(P) \}$ and some
    sufficiently large constant $\kappa \gg 0$. For $\WT := c$, we then get
    for all $v \neq \vopt$
    \[
        \arb(v) \ = \ \argmax \left\{ \frac{c^t(u-v)}{\NO(u-v)} : u \in
        \Nb{P,c}{v}  \right\} 
    \]
    and hence $\arb = \arb_{P,c}^\NO(\WT)$.
\end{proof}

\begin{proof}[Proof of Theorem~\ref{thm:NW_universality}]
    Let $P$ be a simple polytope given by $P = \{ x : Ax \le b \}$ for some
    matrix $A$ and vector $b$. Simplicity implies that for every $A'$ there
    is an $\eps > 0$ such that $P' := \{ x : (A+\eps A')x \le b\}$
    is combinatorially isomorphic to $P$. It is straightforward to verify that
    for a sufficiently general $A'$, the polytope $P'$ is edge-generic. 
    Let $\arb_1,\dots,\arb_s$ be the $c$-arborescences produced by the given
    memory-less pivot rule and let $c_1,\dots,c_s$ objective functions
    such that $\arb_i$ was produced for $(P',c_i)$. Let $\NO_i$ be the
    normalization of Proposition~\ref{prop:edge_gen} and let $R$ be the
    NW-rule with $\WT^R(P',c_i) = c_i$ and $\NO^R(P',c_i) = \NO_i$. It follows
    from Proposition~\ref{prop:edge_gen} that $\arb^R_{P',c_i} =
    \arb_{P',c_i}^{\NO_i}(c) = \arb_i$, which proves the claim.
\end{proof}

\section{Two constructions: Proof of Existence Theorems \ref{thm:PP} and \ref{thm:NP}} \label{sec:PP+NP}

We prove the existence of the two polytopes parametrizing NW pivot rules.
They correspond to Theorems \ref{thm:PP} and \ref{thm:NP}.

\subsection{Pivot rule polytopes} \label{sec:pivot_polytopes}

Let $(P,c)$ be a fixed linear program and $\NO$ a normalization. In this
section we prove Theorem~\ref{thm:PP}, which provides a complete answer to question (A) from the introduction: 

\begin{center}
    \it How does the arborescence of a memory-less pivot rule change when the weight $\WT$ changes? 
\end{center}

For an arborescence $\arb$ of $G(P)$ we define 
\begin{equation}\label{eqn:PPGKZ}
    \PPGKZ^{\NO}(\arb) \ := \ \sum_{v} \frac{\arb(v) - v}{\NO(\arb(v) -v)} \,,
\end{equation}
where we tacitly declare $\frac{0}{\NO(0)} = 0$. The \Def{pivot polytope} of
$(P,c)$ and a fixed normalization $\NO$ is defined as
\begin{equation}\label{eqn:PP}
    \PP_{P,c}^\NO \ := \ \conv \{ \PPGKZ^{\NO}(\arb) : \arb \text{
        $c$-arborescence of $(P,c)$} \} \, .
\end{equation}

We remind the reader that for $\WT$, the arborescence of $(P,c)$ determined
by~\eqref{eqn:NW} is denoted by $\arb_{P,c}^\NO(\WT)$. We can now prove
Theorem~\ref{thm:PP}. Recall that for polytopes $P_1 = \conv(V_1)$ and
\mbox{$P_2 =
\conv(V_2)$}, the \Def{Minkowski sum} is the polytope
\[
    P_1+P_2 \ = \  \{ p_1 + p_2 : p_1 \in P_1, p_2 \in P_2 \} \ = \ 
    \conv( v_1 + v_2 : v_1 \in V_1, v_2 \in V_2 ) \, .
\]

\begin{proof}[Proof of Theorem~\ref{thm:PP}]
    For a vertex $v \neq \vopt$ define
    \begin{equation}\label{eqn:PPv}
        \PP_{P,c}^\NO(v) \ := \ \conv \left\{ \frac{u-v}{\NO(u-v)} : u \in
        \Nb{P,c}{v} \right\} \, .
    \end{equation}
    It follows from the definition of Minkowski sums that
    \begin{equation}\label{eqn:PP_sum}
        \PP_{P,c}^\NO \ = \ \sum_{v \neq \vopt} \PP_{P,c}^\NO(v) \, .
    \end{equation}
    For a generic weight $\WT \in \R^d$ note that 
    \[
        (\PP_{P,c}^\NO)^\WT \ = \ \sum_{v \neq \vopt} (\PP_{P,c}^\NO(v))^\WT \,
        .
    \]
    Hence $(\PP_{P,c}^\NO)^\WT$ is a vertex if and only if
    $(\PP_{P,c}^\NO(v))^\WT$ is a vertex for all $v \neq \vopt$. Now,
    $\frac{u_*-v}{\NO(u_*-v)}$ is this vertex if and only if 
    $\frac{\WT^t(u_*-v)}{\NO(u_*-v)} > \frac{\WT^t(u-v)}{\NO(u-v)}$ for all
    $u \in \Nb{P,c}{v} \setminus u_*$. Set $\arb(v) := u_*$ and $\arb(\vopt) :=
    \vopt$. It now follows from~\eqref{eqn:NW} that $\PPGKZ^\NO(\arb) =
    (\PP_{P,c}^\NO)^\WT$ if and only if $\arb = \arb_{P,c}^\NO(\WT)$, which
    proves the claim.
\end{proof}

\subsection{Neighbotopes}\label{sec:neighbotopes}

Let $P \subset \R^d$ be a polytope and $c$ a generic objective function that
induces a unique sink orientation on the graph $G(P)$ with optimum $\vopt$. The
question of basic \emph{parametric linear programming} is for which objective
functions $c'$ will $\vopt$ be the sink. Geometrically, this is given by the
interior of the normal cone $\Ncone_P(\vopt) = \{ y : y^t \vopt > y^t u
\text{ for all } u \in \Nb{P}{\vopt} \}$. The collection $\Ncone_P = \{
\Ncone_P(v) : v \in V(P) \}$ gives rise to the \Def{normal fan} of $P$, whose
cones give a conical subdivision of $\R^d$.

A more refined question is which $c'$ yield the same unique sink orientation
as $c$. Obviously $c'$ has to satisfy $(c')^t u > (c')^t v$ for all edges $uv
\in E(P)$ such that $c^tu > c^tv$, which defines the interior of a polyhedral
cone. The collection of these cones for varying $c$ again yield a fan
structure, that is the normal fan of a polytope. To be precise, we define the
\Def{edge zonotope} (or \Def{EZ-tope})
\[
    \EZ(P) \ := \ \sum_{uv \in E(P)} [u-v,v-u]
\]
and it is straightforward to verify that the vertices of $\EZ(P)$ are in
bijection to unique sink orientations induced by objective functions. The
EZ-tope was introduced by Gritzmann--Sturmfels~\cite{GS} under the name
\emph{edgotope}.

We address our second primary question of the paper:
\begin{center}
    \emph{Given a fixed $P$ and NW-rule, how does the arborescence change when
    $c$ is varied?}
\end{center}

We will answer it but make the following assumption on the NW-rule $R$:
\begin{compactenum}[\rm i)]
    \item The normalization function does not depend on $c$: $\NO^R(P,c) =
        \NO^R(P,c')$ for all $c,c'$;
    \item The rule $R$ chooses $c$ as the weight: $\WT^R(P,c) = c$.
\end{compactenum}

For example, greatest-improvement as well as $p$-steepest-edge belong to this
class but max-slope with normalization $\NO^\MS(u-v) = c^t(u-v)$ does not. To
stress the two requirements above, we write $\br_{P}^\NO(c) :=
\arb_{P,c}^\NO(c)$ for the arborescence obtained from the linear program
$(P,c)$ with respect to the NW-rule with normalization $\NO$ and weight $\WT =
c$. If $\br = \br_{P}^\NO(c)$, we note that for all $v \in V(P)$ 
\begin{equation}\label{eqn:Uarb}
    \br(v) \ = \ \argmax \left\{ \frac{c^t(u-v)}{\NO(u-v)} : u \in
    N(P,v) \cup \{v\} \right\} \, .
\end{equation}
Indeed, let us denote by $\vopt$ the unique sink of $P$ with respect to $c$.
For $v \neq \vopt$ there is a neighbor $u \in \Nb{P}{v}$ with $c^tu > c^tv$ and
the right-hand side of~\eqref{eqn:Uarb} coincidences with~\eqref{eqn:NW}.
If $v = \vopt$, then the maximum is attained at $u = v$ and we get $\br(v) = v$.

Let $G(P)$ be the undirected graph of $P$. An \Def{arborescence} on $G(P)$
is a map $\br : V(P) \to V(P)$ such that 
\begin{compactenum}[(a)]
    \item there is a unique $\vopt \in V(P)$ with $\br(\vopt) = \vopt$,
    \item for all $v \neq \vopt$ we have $\br(v) \in \Nb{P}{v}$, and
    \item for all $v$ there is $k \ge 1$ such that $\br^k(v) = \vopt$.
\end{compactenum}
In particular, every $c$-arborescence of $(P,c)$ is an arborescence;
cf.~Section~\ref{sec:LPs}.

Consistently, we define for an arborescence $\br$
\[
    \PPGKZ^{\NO}(\br) \ := \ \sum_{v} \frac{\br(v) - v}{\NO(\br(v)
    -v)} \, ,
\]
and we define the \Def{neighbotope} of $P$ for the normalization $\NO$ 
\begin{equation}\label{eqn:N}
    \NP[\NO]{P} \ := \ \conv \{ \PPGKZ_P^\NO(\br) : \br \text{ arborescence of
        $G(P)$} \} \, .
\end{equation}
Let us emphasize that the neighbotope is defined in terms of \emph{all}
arborescences of the undirected graph $G(P)$.

\begin{proof}[Proof of Theorem~\ref{thm:NP}]
    The proof is along similar lines as that of Theorem~\ref{thm:PP}. For a
    vertex $v \in V(P)$ we define
    \begin{equation}\label{eqn:Nhood}
        \NP[\NO]{P}(v) \ := \ \conv \left\{ \frac{u-v}{\NO(u-v)} : u \in
        \Nb{P}{v} \cup \{v\} \right\} \, .
    \end{equation}
    Let $c$ be a generic objective function and let $\br = \br_{P}^\NO(c)$.
    For $v \in V(P)$ it follows directly from~\eqref{eqn:Uarb} that 
    \[
        \NP[\NO]{P}(v)^c = \frac{u-v}{\NO(u-v)} \quad \text{ if and only if }
        \quad \br(v) = u \, .
    \]
    Hence the coherent arborescences are precisely the vertices of 
    \[
        Q := \sum_{v \in V(P)} \NP[\NO]{P}(v) \, ,
    \]
    which is the convex hull over all $\PPGKZ^\NO(f)$ where $f$ ranges over all
    maps $f : V(P) \to V(P)$ with $f(v) \in \Nb{P}{v}$ for all $v \in V(P)$.
    However, the above argument shows that we can discard those $f$ that are
    not arborescences of $G(P)$ and hence $Q = \NP[\NO]{P}$ as claimed.
\end{proof}

The structural similarity between pivot polytopes and neighbotopes can be made
more precise.

\begin{cor}\label{cor:PP_NP}
    Let $P \subset \R^d$ be a polytope, and let $\NO$ be a normalization.  Then the
    neighbotope $\NP[\NO]{P}$ is given by
    \[
        \NP[\NO]{P} \ = \  \conv\left(\bigcup_{c \in \R^{n}}
        \PP_{P,c}^{\NO}\right) \, .
    \]
\end{cor}

So far we have presented two constructions, which help classify and organize
all pivot rules of a linear program. We will now present some examples to
illustrate the construction and, at the same time, highlights the  incredibly
rich combinatorics that the constructions bring to light.

\section{Examples of pivot rule polytopes and neighbotopes}\label{sec:examples}

Let us begin with three examples that illustrate the richness of pivot rule polytopes:

\newcommand\1{\mathbf{1}}%
\begin{example}[$\GI$- and $\pSE$-Pivot polytopes of simplices]\label{ex:simplex1}
    Let $\Delta_{d-1} = \conv( e_1,\dots, e_d) \subset \R^d$ be the standard
    $d$-simplex. An objective function $c$ is generic for $\Delta_{d-1}$ if
    and only if $c_i \neq c_j$ for all $i \neq j$. Up to symmetry, we may
    assume that $c_1 < c_2 < \cdots < c_d$. Observe that $\NO^{\pSE}(e_i -
    e_j) = ||e_{i} - e_{j}||_{p} = 2^{1/p}$ for all $i \neq j \in [n]$, which
    implies that the pivot rule polytopes for the greatest-improvement and
    $p$-steepest-edge normalizations are the same up to scaling.  Thus, it
    suffices to focus on the greatest-improvement normalization $\NO^\GI
    \equiv 1$.
    
    An arborescence of $(\Delta_{d-1},c)$ can be identified with a map $\arb :
    [d] \to [d]$ with $\arb(d) = d$ and $\arb(j) > j$ for all $j < d$.  There
    are precisely $(d-1)!$ arborescences, since there are $d-j$ independent
    choices of an outgoing edge for each $j$. However, not all of these
    arborescences will necessarily arise from $\GI$-rules. To characterize
    those that do, choose $\WT \in \R^d$ such that $\WT_i \neq \WT_j$ for all
    $i \neq j$. We can associate to $\WT$ the permutation $\sigma$ such that $
    \WT_{\sigma^{-1}(1)} < \WT_{\sigma^{-1}(2)} < \cdots <
    \WT_{\sigma^{-1}(d)}$. This permutation uniquely identifies the
    arborescence in the sense that $\WT'$ yields the same coherent
    arborescence as $\WT$ if and only if $ \WT'_{\sigma^{-1}(1)} <
    \WT'_{\sigma^{-1}(2)} < \cdots < \WT'_{\sigma^{-1}(d)}$.
    A \Def{left-to-right maximum} of $\sigma =
    (\sigma_1,\sigma_2,\dots,\sigma_d)$ is an index $j$ such that $\sigma_i <
    \sigma_j$ for all $i < j$. Let $1 \le j_1 < j_2 < \cdots < j_k \le d$ be
    the positions of left-to-right maxima.  It follows from~\eqref{eqn:NW}
    that the coherent arborescence with respect to $\WT$ is given by $\arb(i)
    = j_s$ if $j_{s-1} \le i < j_s$, where we set $j_0 = 0$. Although all
    possible subsets of $[d]$ can occur as positions of left-to-right maxima,
    the position $1$ is never relevant.  Therefore, there are exactly
    $2^{d-2}$ coherent arborescences.
    
    For an arborescence $\arb$, let $\delta_i := |\arb^{-1}(i)|$, so that 
    $\delta(\arb) = (\delta_1,\dots,\delta_d)$ is the \Def{in-degree sequence}
    of $\arb$. It now follows from~\eqref{eqn:PPGKZ} that
    $\PPGKZ^\GI(\arb) = \delta - \1_{[d-1]}$ and hence
    \[
        \PP_{\Delta_{d-1},c}^\GI + \1_{[d-1]} \ = \ \conv\{ \delta(\arb) :
        \arb \text{ arborescence of } (\Delta_{d-1},c) \} \, .
    \]
    From~\eqref{eqn:PP_sum}, we infer that 
    \[
        \PP_{\Delta_{d-1},c}^\GI + \1_{[d-1]} \ = \ \sum_{i=1}^{d-1} \conv\{
            e_{i+1},\dots,e_{d} \} \, .
    \]
    Following the exposition~\cite[Sect.~8.5]{PostPerm}, this shows that
    $\PP_{\Delta_{d-1},c}^\GI$ is the \emph{Pitman--Stanley
    polytope}~\cite{PitmanStanley}.
\end{example}

\begin{example}[Pivot rule polytopes of cubes]\label{ex:cubes}
    Let $C_d = [0,1]^d$ be the $d$-dimensional standard cube. Up to
    symmetry, we can assume that a generic objective function $c$ satisfies
    $0 < c_1 < \cdots < c_d$.  We can identify vertices of $C_d$ with
    characteristic vectors $\1_J \subseteq \{0,1\}^d$ for $J \subseteq [d]$.
    In particular, $\Nb{C_d,c}{\1_J} = \{ \1_{J \cup k} : k \not \in J \}$ and
    for $u \in \Nb{C_d,c}{\1_J}$, we have $u - \1_J = e_k$ for some $k \not
    \in J$. This again shows that the pivot polytopes for greatest improvement
    and $p$-steepest-edge are identical. For the max-slope normalization, it
    follows from~\eqref{eqn:PPGKZ} that $\PP_{C_d,c}^\MS$ is linearly
    isomorphic to $\PP_{C_d,c}^\GI$ with respect to the linear transformation
    $x \mapsto \mathrm{diag}(c_1,\dots,c_d)x$. Thus, we only consider the
    pivot polytope for greatest improvement.

    An arborescence can be identified with a map $\arb : 2^{[d]} \to
    2^{[d]}$ with $\arb([d]) = [d]$ and $\arb(J) = J \cup \{i\}$ for
    some $i \in [d] \setminus J$. Since all choices are independent,
    the total number of arborescences is
    \[
        \prod_{J} 2^{d - |J|} \ = \ 
        \prod_{i=0}^d  (2^i)^{\binom{d}{i}} \ = \ 2^{d\cdot 2^{d-1}} \, .
    \]
    Let $\WT \in \R^d$ be a generic weight. We can again assume that there is
    a unique permutation $\sigma$ such that $\WT_{\sigma^{-1}(1)} <
    \WT_{\sigma^{-1}(2)} < \cdots < \WT_{\sigma^{-1}(d)}$. The corresponding
    coherent arborescence $\arb$ then satisfies that $\arb(J) = J\cup k$
    whenever $\sigma(k) > \sigma(h)$ for all $h \not\in J \cup k$. To see that
    every such arborescence $\arb$ determines a unique permutation $\sigma$,
    we set $\sigma(d) := \arb(\emptyset)$ and $\sigma(k) :=
    \arb(\{\sigma(k+1), \dots, \sigma(d)\})$ for $1 \le k < d$. This
    establishes a bijection between $d$-permutations $\sigma$ and coherent
    arborescences $\arb_\sigma$ of $(C_d,c)$ for the greatest improvement
    normalization. If $\sigma = (1,2,\dots,d)$, then
    $\PPGKZ^\GI(\arb_\sigma)_k$ is the number of proper subsets
    $J\subseteq [d]$ such that $\max([d]\setminus J) = k$. Thus
    $\PPGKZ^\GI(\arb_\sigma) = (1,2,\dots,2^{d-1})$. For any other
    permutation $\sigma'$ one observes that $\PPGKZ^\GI(\arb_{\sigma'})
    = \sigma'(\PPGKZ^\GI(\arb_\sigma))$. Hence
    \[
        \PP_{C_d,c}^\GI \ = \ \conv \{ (
        2^{\sigma(1)-1},
        2^{\sigma(2)-1},
        \dots,
        2^{\sigma(d)-1}) : \sigma \text{ $d$-permutation} \} \, 
    \]
    is a \Def{permutahedron}; cf.~\cite{PostPerm}. The pivot polytope for
    $C_3$ together with the corresponding arborescences is depicted in
    Figure~\ref{fig:cube}.
    \begin{figure}[h]
        \centering
        \includegraphics[scale=.9]{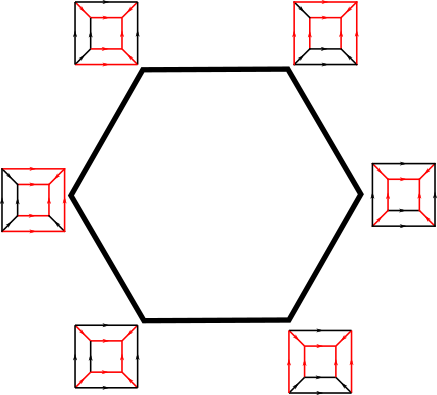}
       \caption{The pivot polytope of $[0,1]^3$ with 
        arborescences associated to the vertices.}
      \label{fig:cube}
   \end{figure}
    We will see a stronger relation in Section~\ref{sec:max_slope}.
\end{example}

Now we present another rich example

\begin{example}[Max-slope pivot polytopes of simplices]\label{ex:simplex_MS}
    Let $P$ be $(d-1)$-dimensional simplex and $c$ a generic linear function.
    We briefly sketch the pivot polytope of $(P,c)$ associated to the
    max-slope normalization $\NO^\MS(u-v) = c^t(u-v)$ and defer the reader
    to~\cite{PivPoly2} for details. Let $v_1,\dots,v_d$ be the vertices of $P$
    labeled such that $c^tv_i < c^tv_j$ if and only if $i < j$. As in
    Example~\ref{ex:simplex1}, we identify an arborescence with a map $\arb :
    [d] \to [d]$ with $\arb(d) = d$ and $\arb(i) > i$ for $i < d$. We call
    such an arborescence \Def{non-crossing} if there are no $i < j$ with
    $j < \arb(i) < \arb(j)$. We show in~\cite{PivPoly2} that an arborescence is
    coherent if and only if it is non-crossing.

    It is easy to see that non-crossing arborescences are in bijection to
    binary trees on $d-1$ nodes. If $d = 1$, then there is a unique
    arborescence that we map to the empty binary tree. For $d > 1$, let $i$ be
    minimal with $\arb(i) = d$. Let $\arb_L : [i] \to [i]$ be defined by
    $\arb_L(a) = \arb(a)$ for $a < i$ and $\arb_L(i) = i$. Further, define
    $\arb_R : [d-i] \to [d-i]$ by $\arb_R(a) := \arb(i+a)-i $. Then $\arb_L$
    and $\arb_R$ are non-crossing arborescences on fewer nodes that yield the
    left and right subtrees of the binary tree associated to $\arb$. Binary
    trees can be equipped with a natural partial order and the resulting poset
    is called the \Def{associahedron}. 

    It is a famous result due to Milnor (unpublished), Haiman (unpublished),
    and Lee~\cite{Lee} that the associahedron is isomorphic to the face
    lattice of a convex polytope. We further show in~\cite{PivPoly2} that
    $\PP_{P,c}^\MS$ is combinatorially isomorphic to the associahedron.
\end{example}

\begin{example}[Max-slope pivot polytopes of prisms over
    simplices]\label{ex:prism_simplex_MS}
    \newcommand\Multip{\mathcal{J}}%
    The associahedron was originally introduced as the poset of partial
    bracketings of a product of $n$ elements in a non-associative
    multiplicative structure. Stasheff's \Def{multiplihedron}
    $\Multip_n$ extends this to the following setup; see~\cite{Sta}. Let
    $f : \mathbf{A} \to \mathbf{B}$ be a morphism between two
    non-associative multiplicative structures. For elements
    $a_1,a_2,\dots,a_{n} \in \mathbf{A}$. What the multiplihedron roughly
    encodes is the possible ways of (partially) evaluating $f(a_1a_2 \cdots
    a_{n})$. Figure~\ref{fig:multiplihedron} gives an example for $n=3$.
    It turns out that the multiplihedron is combinatorially isomorphic to
    the max-slope pivot rule polytope for the prism over the simplex. More
    precisely, if $P = \Delta_{n-1} \times \Delta_1$ and $c = (c_1 < c_2 <
    \cdots < c_{n-1} < c_{n})$ is any linear function, then $\PP_{P,c}^\MS$
    is combinatorially isomorphic to the multiplihedron $\Multip_n$. 
    The relation between of max-slope arborescences of products of
    simplices and non-associative structures will be the main subject
    of~\cite{PivPoly2}.
\begin{figure}[h]
    \centering
    \begin{tikzpicture}
    \node[anchor=south west,inner sep=0] (image) at (0,0)
        {\includegraphics[width=8cm]{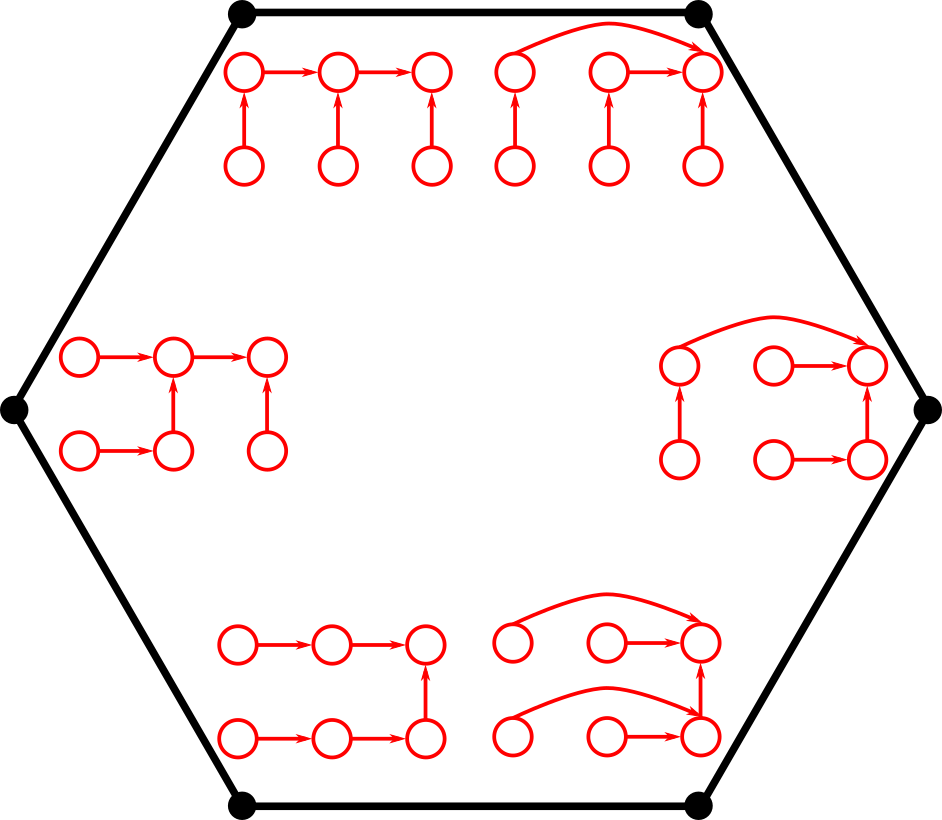}};
    \begin{scope}[x={(image.south east)},y={(image.north west)}]
        \draw[ultra thick,rounded corners] (0.36,0.74) 
        node  {$\scriptstyle (f(a)f(b))f(c)$};
        \draw[ultra thick,rounded corners] (0.65,0.74) 
        node  {$\scriptstyle f(a)(f(b)f(c))$};

        \draw[ultra thick,rounded corners] (0.19,0.38) 
        node  {$\scriptstyle f(ab)f(c)$};
        \draw[ultra thick,rounded corners] (0.82,0.38) 
        node  {$\scriptstyle f(a)f(bc)$};

        \draw[ultra thick,rounded corners] (0.36,.055) 
        node  {$\scriptstyle f((ab)c)$};
        \draw[ultra thick,rounded corners] (0.65,0.055) 
        node  {$\scriptstyle f(a(bc))$};
        \draw[ultra thick,rounded corners] (0,0) 
        node  {\ }; %
    \end{scope}
\end{tikzpicture}
    \caption{The $2$-dimensional multiplihedron $\Multip_3$ and the
    corresponding max-slope arborescences for $\Delta_2 \times \Delta_1$}
    \label{fig:multiplihedron}
\end{figure}
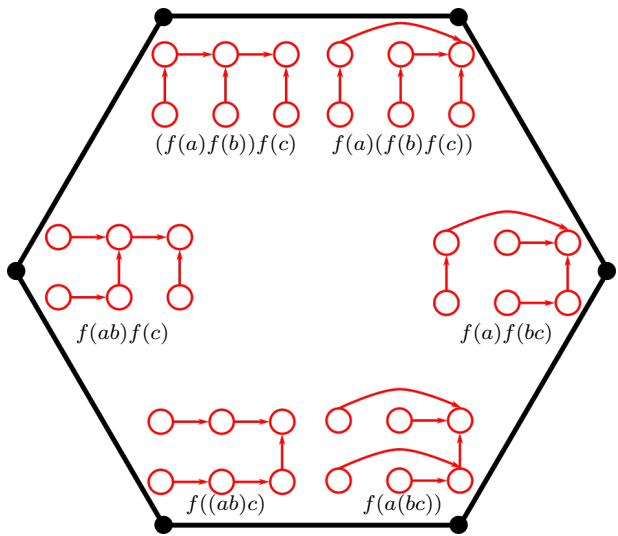
\end{example}

\begin{example}[Greatest-improvement neighbotope of the cube]
    \label{ex:NP_cube}
    Let $C_d = [0,1]^d$ be the unit cube. As in Example~\ref{ex:cubes}, we
    observe that the neighbotope for $\NO$ will be homothetic to
    $\NP[\GI]{C_d}$ if the normalization satisfies $\NO(\pm e_i) =
    \text{const}$ for all $i$. 

    To get the number of arborescences of the $d$-cube, we observe that every
    arborescence $\br$ is given by a spanning tree of $G(C_d)$ together with
    the choice of a root $\vopt \in V(C_d)$. The arborescence is then obtained
    by directing edges of the spanning tree towards the root. The number of
    spanning trees $\tau(C_d)$ of $C_d$ can be found
    in~\cite[Example~5.6.10]{StanleyEC2} and gives the number of arborescences
    \[
        2^n \cdot \tau(C_d) \ = \ \prod_{k=1}^d (2k)^{\binom{d}{k}} \ = \ 
        2^{2^d-1}\prod_{k=1}^d k^{\binom{d}{k}} \, .
    \]
    For a vertex $\1_J \in \{0,1\}^d$, we have
    \[
        \NP[\GI]{C_d}(\1_J) \ = \ \conv( \{ -e_i : i \in J \} \cup \{0\}
        \cup \{ e_i : i \not \in J \}) \, .
    \]
    Let $S = (\Z/2\Z)^d \cong \{-1,+1\}^d$ be the group of sign flips. Every
    element is of the form $s = \1 - 2 \1_J$ for some $J \subseteq [d]$ and
    $\NP[\GI]{C_d}(\1_J) = s \cdot \NP[\GI]{C_d}(\1_\emptyset)$. Thus
    \[
        \NP[\GI]{C_d} \ = \  \sum_{s \in S} s \cdot
        \NP[\GI]{C_d}(\1_\emptyset) \, .
    \]
    Let $W$ be the reflection group of type $B_d$, which acts on $\R^d$ by
    signed permutations. Since $\NP[\GI]{C_d}(\1_\emptyset)$ is invariant
    under permutations, $\NP[\GI]{C_d}$ is invariant with respect to $W$.
    Thus, for a given objective function $c$, we may assume that $0 < c_1 <
    \cdots < c_d$ and it follows from Corollary~\ref{cor:PP_NP} that 
    \[
        (\NP[\GI]{C_d})^c \ = \ (\PP_{P,c}^{\GI})^c \ = \
        (1,2,\dots,2^{d-1}) \, .
    \]
    This shows that $\NP[\GI]{C_d}$ is the \Def{type-B permutahedron} with
    respect to the point $(1,2,\dots,2^{d-1})$, which has $d!2^d$ vertices.
\end{example}

\begin{example}[Neighbotopes of cross-polytopes]\label{ex:NP_octahedron}
    \newcommand\Cross{C^*}%
    The $d$-dimensional \Def{cross-polytope} is the non-simple polytope
    $\Cross_d = \conv\{ \pm e_i : i=1,\dots,d\}$. For $v = se_i$ with $s \in
    \{-1,+1\}$
    \[
        \NP[\GI]{\Cross_d}(v) \ = \ \conv( \{ \pm e_j : j \neq i \} \cup
        \{se_i\}) \, ,
    \]
    which is a pyramid over $\Cross_{d-1}$. The cross-polytope is also
    invariant under the group $W$ of signed permutations. Hence, we may again
    assume that $0 < c_1 < \cdots < c_d$ and the corresponding arborescence
    $\br = \br_{\Cross_d}^\NO(c)$ satisfies $\br(v)  =  e_d$ if $v \neq -e_d$
    and $\br(-e_d) = e_{d-1}$. It follows that $\NP[\GI]{\Cross_d}$ is the
    type-$B$ permutahedron for the point $(2d-1)e_d + e_{d-1}$ and has $4
    d(d-1)$ vertices.
\end{example}

\section{The Combinatorics of Pivot rule polytopes and Neighbotopes} \label{sec:facesPP+NP}

We investigate the basic combinatorial questions on polyhedra for our
constructions and the relation to fiber polytopes and sweep polytopes.

\subsection{Faces of pivot rule polytopes}\label{sec:PP_facial}
Before we discuss general faces of pivot rule polytopes, we look at
vertices and their numbers. 

We recall that a $d$-dimensional polytope $P$ is \Def{simple} if every
vertex is incident to $d$ edges. For a simple $d$-polytope $P \subset \R^d$
and generic objective function $c$, denote by $h_i$ the number of vertices
with in-degree $i$. Since $P$ is simple, $h_i$ is independent of $c$ and
$h(P) = (h_0,\dots,h_d)$ is the \Def{$h$-vector} of $P$;
cf.~\cite[Ch.~6]{barvinok}. 

\begin{prop}[{\cite[Prop.~3.1]{enumerative}}] \label{prop:UB}
    For a simple $d$-polytope $P$ and a generic objective function, the
    total number of arborescences is given in terms of the entries of the $h$-vector by
    \[
        1^{h_1}
        2^{h_2}
        \cdots
        d^{h_d} \, .
    \]
\end{prop}

We now show that this upper bound cannot be attained for coherent
arborescences, independent of the normalization.
\begin{theorem}\label{thm:PP_UB}
    Let $P \subset \R^d$ be a simple $d$-polytope with $n > d+1 \ge 4$ vertices and
    $h$-vector $h(P) = (h_0,\dots,h_d)$. For fixed objective function $c$
    and arbitrary normalization $\NO$, the number of coherent arborescences
    is strictly less than 
    \[
        1^{h_1} 2^{h_2} \cdots d^{h_d} - 2 (n-m-2) \, ,
    \]
    where $m$ is the number of facets of $P$.
\end{theorem}
\begin{proof}
    We need to bound the number of vertices of $\PP_{P,c}^\NO$. Recall
    from~\eqref{eqn:PP_sum} that $\PP_{P,c}^\NO$ is a Minkowski sum of
    polytopes $\PP_{P,c}^\NO(v)$ for $v\neq \vopt$. Since $P$ is simple, the
    polytopes $\PP_{P,c}^\NO(v)$ are all simplices of various dimensions.
    Using the interpretation of the $h$-vector given above, we see that the
    number of $(k-1)$-simplices is precisely $h_k$. In particular, we have
    $h_0 = 1$ and $h_1 = m - d$ vertices with in-degree $1$, where $m$ is
    the number of facets.  Since $P$ is not a simplex, we have $n \ge
    (d+1)(d-2) + m(d-1)$ by the Lower Bound Theorem (cf.~\cite{ziegler}). Thus
    $\PP_{P,c}^\NO$ is a Minkowski sum of $ N := n - (m-d + 1) \ge
    (d+2)(d-2) + 1 + m(d-2) \ge d+2$ simplices of positive dimension. Let
    $v_1,\dots, v_N$ be the corresponding vertices and set $\PP_i :=
    \PP_{P,c}^\NO(v_i)$.

    \newcommand\C{\mathcal{C}}%
    We slightly extend the argument from~\cite[Sect.~6]{sanyal09}: For $u
    \in V(\PP_i)$, let $\Ncone_{\PP_i}(u)$ be set of linear functions $\WT$
    such that $\{u\} = \PP_i^\WT$. This is a non-empty open polyhedral cone.
    For $u_i \in V(\PP_i)$, we have that $\sum_i u_i$ corresponds to a
    vertex of $\PP_{P,c}^\NO$ if and only if $\bigcap_i \Ncone_{\PP_i}(u_i)
    \neq \emptyset$.  Fix $u_i \in V(\PP_i)$ for $i=d+2,\dots,N$ and assume
    that for all choices of $u_j \in V(\PP_j)$ for $j =1,\dots,d+1$,
    $\sum_{i=1}^N u_i$ corresponds to a vertex. For $1 \le j \le d+1$,
    define $\C_{j}$ to be the collection of open convex sets
    \[
        \Ncone_{\PP_j}(u)  \cap 
        \Ncone_{\PP_{d+2}}(u_{d+2})  \cap 
        \cdots \cap
        \Ncone_{\PP_{N}}(u_{N})
    \]
    for $u \in V(\PP_j)$. The sets in $\C_{j}$ are pairwise disjoint.  Since $C_1 \cap \cdots \cap C_{d+1} \neq \emptyset$
    for all choices $C_j \in \C_{j}$, $j=1,\dots,d+1$, Lov\'asz'
    colorful Helly Theorem (cf.~\cite{Barany}) implies that $\bigcap_{C \in
    \C_{j}} C \neq \emptyset$ for some $j$, which yields a
    contradiction. There are at least $2(N-(d+1))$ choices of vertices
    $u_{d+2},\dots,u_{N}$, which finishes the proof. 
\end{proof}

A precise but more involved bound can be obtained from the Minkowski Upper
Bound Theorem~\cite{MUBT}. If $P$ is a $d$-dimensional simplex, then $h(P) =
(1,\dots,1)$.  Neither Theorem~\ref{thm:PP_UB} nor the Minkowski Upper Bound
Theorem rules out the possibility, that $P$ has $ 1^{h_1} 2^{h_2} \cdots
d^{h_d} = d!$ many coherent arborescences.

\begin{quest}
    For every $d \ge 1$, is there a normalization $\NO$ such that all
    arborescences of $(\Delta_{d},c)$ are coherent?
\end{quest}

In the rest of the section we discuss the facial structure of the pivot
polytope $\PP_{P,c}^\NO$ for the LP $(P,c)$ and a normalization $\NO$. For
this we assume that $P$ is a \emph{simple} polytope.  A
\Def{$c$-multi-arborescence} is a map $\arb : V(P) \to 2^{V(P)} \setminus
\{\emptyset\}$ that satisfies $\arb(\vopt) = \{\vopt\}$ and $\arb(v) \subseteq
\Nb{P,c}{v}$ for all $v \neq \vopt$. We will abuse notation and identify
one-element subsets of $V(P)$ with the elements themselves. Hence we can view
$c$-arborescences as $c$-multi-arborescences with $|\arb(v)| = 1$ for all $v
\in V(P)$. If $\WT \in \R^d$ is a non-generic weight, then the maximum
in~\eqref{eqn:NW} may not be uniquely attained for all $v$ and gives rise to
\Def{coherent} $c$-multi-arborescence that we will also denote by
$\arb_{P,c}^\NO(\WT)$.

Given two multi-arborescences $\arb, \arb'$, we say that $\arb$ \Def{refines}
$\arb'$, written $\arb \preceq \arb'$, if $\arb(v) \subseteq \arb'(v)$ for all
$v \in V(P)$. This is a partial order on the collection of multi-arborescences
of $(P,c)$. The proof of Theorem~\ref{thm:PP} yields the facial structure of
pivot polytopes.

\begin{theorem}\label{thm:PP_faces}
    Let $P \subset \R^d$ be a simple polytope, $c$ a generic objective
    function, and $\NO$ a normalization. For two weights $\WT,\WT'$ we have
    \[
        \arb(\WT) \ \preceq \ \arb(\WT')  
        \quad \Longleftrightarrow \quad
        (\PP_{P,c}^\NO)^\WT \ \subseteq \ (\PP_{P,c}^\NO)^{\WT'} \, .
    \]
    Thus, the poset of coherent arborescences is isomorphic to the face
    lattice of $\PP_{P,c}^\NO$.
\end{theorem}

Two arborescences $\arb,\arb'$ differ by an \Def{edge rerouting} if there is a
unique vertex $v \in V(P)$ with $\arb(v) \neq \arb'(v)$. As a consequence, we
get a necessary condition for the adjacency of two coherent arborescences.

\begin{cor}\label{cor:PP_reroute}
    If the vertices of $\PP_{P,c}^\NO$ corresponding to two coherent
    arborescences $\arb,\arb'$ are adjacent, then $\arb,\arb'$ differ by an
    edge rerouting.
\end{cor}

Note that the definition of $\PP_{P,c}^\NO$ given in~\eqref{eqn:PP} involves
\emph{all} arborescences. If $\arb$ is a non-coherent arborescence, the
geometry of $\PP_{P,c}^\NO$ gives us to the finest coherent coarsening of
$\arb$.

\begin{prop}\label{prop:PP_coarsening}
    Let $\arb$ be a $c$-arborescence of $(P,c)$ and let $F \subseteq
    \PP_{P,c}^\NO$ be the unique face containing $\PPGKZ^\NO(\arb)$ in
    its relative interior. Then the $c$-multi-arborescence $\arb'$ associated to
    $F$ is the finest coherent coarsening of $\arb$.
\end{prop}

\begin{proof}
    Let $\WT$ be a weight such that $F = (\PP_{P,c}^\NO)^\WT$. For every $v
    \neq \vopt$, let $F_v = (\PP_{P,c}^\NO(v))^\WT$. Then $F = \sum_v F_v$
    and, in particular $\frac{\arb(v) - v}{\NO(\arb(v)-v)} \in F_v$ for all $v
    \neq \vopt$. The multi-arborescence associated to $F$ is given by
    \[
        \arb'(v) \ = \ \left\{ u \in N_P(c,v) : \tfrac{\arb(v) - v}{\NO(\arb(v)-v)}
        \in F_v \right \}
    \]
    and hence is a coarsening of $\arb$. If $\arb''$ is another coherent
    multi-arborescence coarsening $\arb$ with corresponding face $G$, then
    $\PPGKZ^\NO(\arb) \in G$ and hence $\PPGKZ^\NO(\arb) \in G
    \cap F$. But since $\PPGKZ^\NO(\arb)$ is contained in the relative
    interior of $F$, it follows that $F \cap G = F$, which shows that
    $\arb'$ is the finest coherent coarsening of $\arb$.
\end{proof}

Let us close by remarking that the poset of \emph{all} $c$-arborescences on
$(P,c)$ can also be realized as the face poset of a convex polytope.

\begin{prop}
    \label{prop:Bauespos}
    Let $P$ a simple polytope with $h$-vector $h(P) = (h_0,\dots,h_d)$. For a
    generic linear function $c$, the poset of all $c$-multiarborescences is isomorphic
    to the face poset of the polytope
    \[
        \Delta_{0}^{h_1} \times 
        \Delta_{1}^{h_2} \times 
        \cdots \times
        \Delta_{d}^{h_d} \, .
    \]
\end{prop}
\begin{proof}
    Recall the face structure of a product of simplices $\prod_{i = 1}^{n}
    \Delta_{d_{i}}$ for $d_{i} \in \mathbb{N}$ is given by a choice of subset
    $S_{i}$ from each $[d_{i}]$. Then two collections of subsets $\{S_{i}\}$ and
    $\{T_{i}\}$ correspond to faces that contain one another precisely when
    $S_{i} \subseteq T_{i}$ for all $i \in [n]$.

    Let $\vopt$ be the unique maximizer of $c$ over $P$.  A
    $c$-multi-arborescence is given by the choice of a non-empty subset
    $S_{v}$ of $c$-improving neighbors for every vertex $v \neq \vopt$.  This
    choice is made independently, so any possible collection of subsets
    corresponds to some $c$-multi-arborescence. The collection of all
    multi-arborescences then corresponds to all choices of sets of outgoing
    neighbors. 
    
    For a generic orientation on a simple polytope, the $h$-vector counts the
    number of vertices with a given out-degree. Hence, the set of all possible
    choices of subsets of outgoing edges is given by
    \[
    \Delta_{0}^{h_1} \times 
    \Delta_{1}^{h_2} \times 
    \cdots \times
    \Delta_{d}^{h_d} \, . \qedhere
    \]
\end{proof}

The space of all monotone paths yields a similar cell-complex called the Baues poset of cellular strings. In general, that complex is not polytopal. Proposition \ref{prop:Bauespos} shows that the analogous poset for arborescences is instead always the face lattice of product of simplices. Furthermore, the choice of simplices is independent of $c$ so long as $c$ is generic, since the h-vector is invariant.

More generally, pivot polytopes and their lattices of coherent
multi-arborescences behave in analogy to fiber polytopes and their lattices of
coherent subdivisions. This analogy is particularly strong in the case of
monotone path polytopes and
secondary polytopes. To start, given a multi-arborescence or subdivision,
evaluating whether it is coherent in its respective context corresponds to
solving a linear feasibility problem. For adjacency in the monotone path
polytope or secondary polytope, adjacent vertices must satisfy that their
corresponding coherent monotone paths or coherent triangulations differ by a
flip. However, differing by a flip is not sufficient to guarantee adjacency. The
edge reroutings thus play the role of flips for pivot polytopes. Furthermore,
given any monotone path or triangulation, we can assign it a canonical
point the unique face containing the point in its relative interior
corresponds to the finest coherent coarsening of the 
monotone path or triangulation respectively. This is precisely the statement
of Proposition~\ref{prop:PP_coarsening} for pivot rule polytopes.

\subsection{Faces of neighbotopes}\label{NP_facial}

As in the previous section we start by understanding the vertices of
neighbotopes.  We can again use the description as a Minkowski sum to derive
an upper bound on the number of coherent arborescences.

\begin{prop}
    Let $P$ be a simple $d$-polytope with $n$ vertices, then the number of
    coherent arborescences of $P$ is at most $d^n (1-\frac{1}{d^{d+1}})$.
\end{prop}
\begin{proof}
    If $P$ is simple, then $\NP[\NO]{P}(v)$ is a $(d-1)$-simplex for all $v
    \in V(P)$. The same argument as in the proof of 
    Theorem~\ref{thm:NP} applies and shows that of the $d^n$ possible
    vertices of the Minkowski sum at least $d^{n-(d+1)}$ fail to be
    vertices.
\end{proof}

Note that this bound is in general far from being tight: If $P$ is a
$d$-simplex, then $\NP[\NO]{P}(v) = -v + P$ for all vertices $v$ and 
$\NP[\NO]{P}$ is homothetic to $P$; see also
Proposition~\ref{prop:NP_equiv}. 
Of course, the Minkowski sum decomposition is also valid for non-simple
polytopes and a more involved upper bound maybe derived in the same manner.

We now consider the face lattice of neighbotopes.
A \Def{multi-arborescence} of a polytope $P$ 
is a map $\br : V(P) \to 2^{V(P)}$ that satisfies 
\begin{compactenum}[(a)]
    \item for all $v$ we have $\br(v) \subseteq \Nb{P}{v} \cup \{v\}$;
    \item there is a unique face $F_\br \subseteq P$ with $\{ v : v \in \br(v)
        \} = V(F_\br)$ and 
    \item $\br(v) = \Nb{P}{v} \cup \{v\}$ for all $v \in
        V(F_\br)$;
    \item for all $v \in V$ there is $k \ge 1$ with $\br^k(v) = V(F_\br)$.
\end{compactenum}

\begin{prop}
    Let $P$ be a polytope and $\NO$ a normalization. Every face of the
    neighbotope $\NP[\NO]{P}$ can be identified with a unique
    multi-arborescence.
\end{prop}

\begin{proof}
    Let $F = (\NP[\NO]{P})^c$ be a face of the neighbotope $\NP[\NO]{P}$.  It
    follows from the proof of Theorem~\ref{thm:NP} and \eqref{eqn:Nhood} that
    $F = \sum_v F_v$, where $F_v = \NP[\NO]{P}(v)^c$. We define a
    multi-arborescence $\br$ by $u \in \br(v)$ if and only if
    $\frac{u-v}{\NO(u-v)} \in F_v$ for all $v \in V(P)$. Unless $v \in F_\br
    := P^c$, there is always an improving neighbor $u \in \br(v)$ and $v \not
    \in \br(v)$. Otherwise, $\br(v) = \Nb{P}{v} \cup \{v\}$. This shows that
    $\br$ satisfies all defining properties of a multi-arborescence. Since $F$
    can be recovered from $\br$, it also shows that every face corresponds to
    a unique multi-arborescence.
\end{proof}

This injection furthermore encodes a partial order. Namely, we say that $\br
\preceq \br'$ for two multi-arborescences $\br$ and $\br'$, when $\br(v)
\subseteq \br'(v)$ for all $v \in V$. This partial order corresponds to the
partial order of the face lattice. Namely, for faces $F_{1}, F_{2}$ of
$\NP[\NO]{P}$ with corresponding multi-arborescences $\br_{1}$ and $\br_{2}$,
we have that $F_{1}$ is a face of $F_{2}$ if and only if $\br_{1} \preceq
\br_{2}$.

\section{Monotone path polytopes and sweep polytopes}\label{sec:MS+SP}

Now we make connections to two famous constructions in geometric combinatorics.

\subsection{Max-slope pivot rule polytopes and monotone path polytopes}\label{sec:max_slope}

The shadow-vertex rule is a well-known and well analyzed pivot rule that in
its usual form does not belong to the class of memory-less rules. We show that
the max-slope rule is a natural generalization that has the benefit of being a
NW-rule. We also show that it is intimately related to the theory of
(coherent) cellular strings on polytopes of
Billera--Kapranov--Sturmfels~\cite{BKS}.

For the setup, let $P \subset \R^d$ be a $d$-polytope and $c$ a generic
objective function. Let $r$ be a vertex of $P$ together with a weight $\WT \in
\R^d$ such that $r = P^\WT$. We seek to find the optimal vertex $\vopt = P^c$.
Define a linear projection $\pi : \R^d \to \R^2$ by $\pi(x) := (c^tx,
\WT^tx)$. By construction $\pi(r)$ and $\pi(\vopt)$ are vertices of the
projection $\pi(P)$. There is a unique path from $\pi(r)$ to $\pi(\vopt)$ that
is increasing with respect to $c$. Since $c$ and $\WT$ where assumed to be
generic, the pre-image of that path is a $c$-increasing vertex-edge path on $P$
from $r$ to $\vopt$. This is called a \Def{shadow-vertex path} from $r$ to
$\vopt$. 

Note that the path is not determined by $r$ but rather by the choice of $\WT$.
Given $r,\WT$, the shadow path can be found with a variant of the simplex
algorithm but it is clear that the procedure outlined above does not yield
pivots for vertices outside the shadow-vertex path. In fact, it does not even
yield a memory-less pivot rule in the sense of Section \ref{sec:intro} for the vertices \emph{on}
the shadow-vertex path as they would require a choice of a weight that might
lead to a different shadow-vertex path. 

The vertices on the shadow-path from $r$ to $\vopt$ can be characterized
locally. Starting from $r$, a $c$-improving neighbor $s \in \Nb{P,c}{r}$ will
be the next vertex on the shadow path if $[\pi(r),\pi(s)]$ is an edge of
$\pi(P)$. This happens if the slope of the edge in the plane is \emph{maximal}
among all edges of $\pi(P)$ incident to $\pi(r)$. That is, if
$\frac{\WT^t(s-r)}{c^t(s-r)} > \frac{\WT^t(u-r)}{c^t(u-r)}$ for all $u \in
\Nb{P,c}{r} \setminus s$. The max-slope rule now extends this condition to all
vertices: if $v \neq \vopt$, then max-slope chooses the neighbor
\begin{equation}\label{eqn:MS}
    u_* \ = \ \argmax \left\{ \frac{\WT^t(u-v)}{c^t(u-v)} : u \in \Nb{P,c}{v}
    \right\} \, .
\end{equation}
Our discussion above now proves the following.
\begin{prop}\label{prop:shadow_path}
    For $(P,c)$, let $\WT$ be a generic weight and $r = P^\WT$. Let $\arb^\MS =
    \arb_{P,c}^\MS(\WT)$ be the max-slope arborescence of $(P,c)$ with respect
    to $\WT$. The path $(r_i)_{i\ge0} \in V(P)$ with $r_0 := r$ and $r_{i+1} :=
    \arb^\MS(r_i)$ for $i \ge 0$ is precisely the shadow-vertex path of $(P,c)$
    with respect to $\WT$.
\end{prop}

Let $\vmopt = P^{-c}$ be the minimizer with respect to $c$.  A \Def{cellular
string} on $(P,c)$ is a sequence of faces $F_0, F_1, \dots, F_k$ with the
property that $\vmopt = F_0^{-c}$, $\vopt = F_k^c$ and $F_{i-1}^c = F_i^{-c}$
for all $i=1,\dots,k$. A cellular string can be refined by replacing some of
the $F_i$ by cellular strings of $(F_i,c)$. This yields a partial order on
cellular strings, called the \Def{Baues poset} of $(P,c)$. The minimal elements
are the \Def{monotone paths} from $\vmopt$ to $\vopt$. 

Billera--Sturmfels~\cite{BSFiberPoly} and
Billera--Kapranov--Sturmfels~\cite{BKS} developed the theory of coherent
cellular strings on polytopes. A monotone path $\vmopt = v_0, v_1,\dots,v_k =
\vopt$ is called \Def{coherent} if there is a $\WT \in \R^d$ such that $v_i$
is the unique maximizer of $\WT$ over the slice $P \cap \{x : c^tx = c^t v_i
\}$ for all $i$. For any monotone path $W = v_0,\dots,v_k$, define the point
\[
    \Phi_W \ := \ \sum_{i=1}^k \frac{c^t(v_i-v_{i-1})}{c^t(v_k-v_{0})}
    (v_i - v_{i-1}) \, 
\]
and with it the \Def{monotone path polytope}
\[
    \Mon_c(P) \ := \  \conv\{ \Phi_W : W \text{ monotone path of } (P,c)
    \} \, .
\]
The vertices of $\Mon_c(P)$ are precisely the coherent monotone paths and,
stronger even, the face lattice of $\Mon_c(P)$ defines the subposet of the
Baues poset of \Def{coherent} cellular strings of $(P,c)$.

We will show next that the max-slope pivot polytope provides a refinement of the
monotone path polytope in the following sense. A polytope $Q \subset \R^d$ is
a \Def{weak Minkowski summand} of a polytope $P \subset \R^d$ if there is
$\lambda >0$ and a polytope $R \subset \R^d$ such that $\lambda P  = Q + R$.
This implies that there is a poset map from the face lattice of $P$ onto the
face lattice of $Q$ with favorable combinatorial and topological properties.
Roughly, this means that the combinatorics of $P$ \emph{refines} the
combinatorics of $Q$.

\begin{prop}\label{prop:MS_Mon}
    Let $P \subset \R^d$ be a polytope and $c$ a generic objective function.
    Then the monotone path polytope $\Mon_c(P)$ is a weak Minkowski summand of
    the max-slope pivot polytope $\PP_{P,c}^\MS$.
\end{prop}
\begin{proof}
    We use a result of Shephard (cf.~\cite[Theorem~15.1.2]{grunbaum}) that
    states that $Q$ is a weak Minkowski sum of $P$ if and only if $Q^\WT$ is a
    vertex whenever $P^\WT$ is a vertex.

    Let $\WT \in \R^d$ be generic and $\arb = \arb_{P,c}^\MS(\WT)$ the
    max-slope arborescence of $(P,c)$ corresponding to the vertex
    $(\PP_{P,c}^\MS)^\WT$.  Let $r = P^\WT$. We can apply the same argument as
    before and obtain a shadow-vertex path from $r$ to $\vmopt$.  Combining
    this path with the shadow-vertex path from $r$ to $\vopt$ yields the
    max-slope path from $\vmopt$ to $\vopt$. Verifying
    condition~\eqref{eqn:MS} along this path then shows that this max-slope
    path is precisely the coherent monotone path induced by $\WT$, which shows
    that $\Mon_c(P)^\WT$ is a vertex.
\end{proof}

\begin{example}[Monotone path polytopes of simplices]\label{ex:Mon_simplex}
    Let $P$ be a $d$-simplex with vertices $v_0, v_1, \dots, v_d$ ordered
    according to a generic objective function $c$.
    In~\cite[p.~545]{BSFiberPoly} it is shown that $\Mon_c(P)$ is
    combinatorially isomorphic to a $(d-1)$-dimensional cube. Any choice $0 =:
    i_0 < i_1 < \cdots < i_{k-1} < i_k := d$ defines a monotone path $v_{i_0},
    v_{i_1}, \dots, v_{i_{k-1}}, v_{i_k}$ and it is straightforward to show that
    every such path is coherent. Continuing Example~\ref{ex:simplex_MS}, we see
    that choosing a non-crossing arborescence for every interval $[i_j,i_{j+1}]
    = \{i_j, i_{j}+1,\dots,i_{j+1}\}$ yields a non-crossing arborescence of
    $(P,c)$ that contains the given monotone path. This shows that the set of
    refinements of a given monotone path to a non-crossing arborescence has the
    structure of a product of associahedra.
\end{example}

Note that we did not require $P$ to be simple in
Proposition~\ref{prop:MS_Mon}. When $P$ is simple,
Proposition~\ref{prop:MS_Mon} yields a necessary criterion when a
multi-arborescence is coherent. Let $\arb$ be a multi-arborescence. For every
$v \in V(P)$ there is a unique smallest face $F_v \subset P$ with $v \cup
\arb(v) \subseteq F$. We can associate to $\arb$ a cellular string as
follows: Set $u_0: = \vmopt$ and $F_0 := F_{u_0}$. For $i \ge 1$, let $u_i :=
F_{i-1}^c$ and $F_i := F_{u_i}$. We call the $c$-multi-arborescence $\arb$
\Def{cellular} if $\arb(v) \subseteq V(F_i)$ for all 
$v \in V(F_i) \setminus u_{i+1}$ and all $i$. That is, if $\arb$ restricts to a $c$-multi-arborescence
of $(F_i,c)$ except for $u_{i+1}$. Note that every $c$-arborescence is
cellular.

\begin{cor}\label{cor:cellular_arb}
    If $\arb$ is a coherent multi-arborescence, then $\arb$ is cellular.
\end{cor}
\begin{proof}
    \newcommand\oP{\bar{P}}%
    \newcommand\ov{\bar{v}}%
    \newcommand\ovopt{\ov_{opt}}%
    \newcommand\ovmopt{\ov_{-opt}}%
    Let $\arb$ be a coherent $c$-multi-arborescence with corresponding face $F =
    (\PP_{P,c}^\MS)^\WT$ for some weight $\WT$. Consider the linear projection
    $\pi : \R^d \to \R^2$ given by $\pi(x) = (c^tx,\WT^tx)$. The polygon $\oP$
    has the two distinguished vertices $\ovmopt = \pi(\vmopt)$ and $\ovopt =
    \pi(\vopt)$ that minimize and maximize the first coordinates,
    respectively. Let $ \ovmopt = \ov_0 \ov_1 \dots \ov_k = \ovopt$ be the
    upper path with respect to the second coordinate. It is straightforward to
    verify that $F_i = \pi^{-1}([\ov_i,\ov_{i+1}])$ yields the cellular string
    as constructed above. If $v \in V(F_i) \setminus u_{i+1}$ then $v' =
    \arb(v)$ if and only if the slope of the segment $[\ov,\ov']$ is maximal
    among all segments $[\ov,\ov'']$ for $\ov'' \in \Nb{P,c}{v}$. Clearly this
    slope is at most that of $[\ov_i,\ov_{i+1}]$ and equal whenever $v' \in
    V(F_i)$. This shows that $\arb$ is cellular.
\end{proof}

Note that there does not seem to be a natural map from 
general multi-arborescences of $(P,c)$ to the Baues poset.

In~\cite[Example~5.4]{BSFiberPoly} it is shown that the monotone path polytope
of the $d$-cube $C_d = [0,1]^d$ with respect to $c = (1,\dots,1)$ is the
permutahedron $\Pi_{d-1} = \conv\{ (\sigma(1),\dots,\sigma(d)) : \sigma \text{
    permutation} \}$. In fact for any generic $c$, $\Mon_c(P)$ will be
combinatorially isomorphic to $\Pi_{d-1}$. It is remarkable that the max-slope
pivot polytope of the cube is also combinatorially isomorphic a permutahedron;
cf.~Example~\ref{ex:cubes}. This is not a coincidence. Recall that a polytope
$P \subset \R^d$ is a \Def{zonotope} if there are $t, z_1,\dots,z_n \in \R^d$
such that 
\[
    -t + P \ = \ [-z_1,z_1] + [-z_2,z_2] + \cdots + [-z_n,z_n] \, .
\]
Moreover, two polytopes $P,Q \subset \R^d$ are \Def{normally equivalent} if
$Q$ is a weak Minkowski summand of $P$ and $P$ is a weak Minkowski summand of
$Q$.

\begin{theorem}\label{thm:MS_zono}
    Let $P$ be a  polytope and $c$ a generic objective
    function.  If $P$ is a zonotope, then $\Mon_c(P)$ and $\PP_{P,c}^\MS$ are
    normally equivalent.
\end{theorem}

\begin{proof}
    In light of Proposition~\ref{prop:MS_Mon}, we only have to show that
    $\Pi_{P,c}^\MS$ is a weak Minkowski summand of $\Mon_c(P)$.
    From~\eqref{eqn:PP_sum}, we see that it suffices to show that
    $\Pi_{P,c}^\MS(v)$ is a weak Minkowski summand of $\Mon_c(P)$ for all $v
    \neq \vopt$.

    We may assume that $P = [-z_1,z_1] + [-z_2,z_2] + \cdots + [-z_n,z_n]$ and
    that $c^tz_i > 0$ for all $i$.  Since $P$ is a linear projection of $C_d =
    [0,1]^d$, it follows from Lemma~2.3 and Theorem~4.1 of~\cite{BSFiberPoly}
    that the monotone path polytope $\Mon_c(P)$ is normally equivalent to the
    zonotope
    \[
        \tilde{\Mon}_c(P) \ = \ \sum_{i < j} [z_i - z_j, z_j - z_i] \, .
    \]
    If $u,v$ are adjacent vertices of $P$, then $u-v = \pm z_j$ for some $j$.
    Thus for $v \neq \vopt$ there is $J \subseteq [n]$ such that 
    \[
        \Pi_{P,c}^\MS(v) \ = \ \conv \{ z_j : j \in J\} \, .
    \]
    For a generic $\WT \in \R^d$ the vertex $\tilde{\Mon}_c(P)^\WT$ is
    determined by the permutation $\sigma$ such that $\WT^t z_{\sigma(1)} >
    \WT^t z_{\sigma(2)} > \cdots > \WT^t z_{\sigma(n)}$. However, this shows
    that $(\Pi_{P,c}^\MS(v))^\WT = z_{\sigma(k)}$, where $k = \min
    \sigma^{-1}(J)$. Hence $(\Pi_{P,c}^\MS(v))^\WT$ is a vertex whenever
    $\tilde{\Mon}_c(P)^\WT$ is, which proves the claim.
\end{proof}

Theorem~\ref{thm:MS_zono} gives a new way of computing monotone path polytopes
of zonotopes. In particular, it says that for every coherent monotone path
there is a \emph{unique} extension to a coherent arborescence.

A polytope $P$ is a \Def{belt polytope}~\cite{Bolker} if $P$ is normally
equivalent to a zonotope. Equivalently, if the normal fan of $P$ is given by a
hyperplane arrangement; cf.~\cite[Ch.~7]{ziegler}.  The next result implies
that Theorem~\ref{thm:MS_zono} can actually be extended to belt polytopes. 

\begin{theorem}
    Let $P,P' \subset \R^d$ be polytopes and $c$ a generic objective function.
    If $P$ is normally equivalent to $P'$, then $\PP_{P,c}^\MS =
    \PP_{P',c}^\MS$.
\end{theorem}

\begin{proof}
    Let $v \in V(P)$. If $v = P^\WT$, then $P^\WT = v'$ is a vertex that is
    independent of $\WT$. This yields a bijection $V(P) \to V(P')$. Moreover,
    it follows from normal equivalence that if $u,v$ are adjacent vertices of
    $P$, then $u',v'$ are adjacent in $P'$ and  $u-v = \lambda( u'-v' )$ for
    some $\lambda > 0$. Thus $u \in \Nb{P,c}{v}$ if and only if 
    $u' \in \Nb{P',c}{v'}$ and $\frac{u-v}{c^t(u-v)} =
    \frac{u'-v'}{c^t(u'-v')}$. Now \eqref{eqn:PPv} yields 
    $\PP_{P,c}^\MS(v) = \PP_{P',c}^\MS(v')$ and the claim follows
    from~\eqref{eqn:PP_sum}.
\end{proof}

\subsection{Neighbotopes and Sweep Polytopes}\label{sec:NP_sweep}

In this section, we relate the neighbotopes with respect to
greatest-improvement pivot rule to another class of well-known polytopes, the sweep polytopes \cite{Sweep}. 
Let $p_1,\dots,p_n \in \R^d$ be a configuration of points. A
permutation $\sigma$ of $[n]$ is called a \Def{sweep} if there is a generic
linear function $c \in \R^d$ such that 
\[
    c^tp_{\sigma^{-1}(1)} < c^tp_{\sigma^{-1}(2)} < \cdots <
    c^tp_{\sigma^{-1}(n)} 
\]
The \Def{sweep polytope}, introduced by Padrol and Philippe in~\cite{Sweep},
captures the sweeps of a point configuration and is defined as
\[
    \SP(p_1,\dots,p_n) \ := \ \sum_{i < j} [p_i-p_j,p_j-p_i]
\]
If $p_1,\dots,p_n$ are the vertices of a polytope $P$, then the sweep is
related to line shellings of the dual to $P$. It was studied in~\cite{GS} under the name \emph{shellotope}.

Recall from the introduction that for a polytope $P \subset \R^d$ and a
normalization $\NO$, the set of \Def{normalized edge directions} is
$\ED^\NO(P) = \{ \frac{u-v}{\NO(u-v)} : uv \in E(P) \}$. If $c$ is a generic
objective function, then $\ED^\NO(P,c) = \{ \frac{u-v}{\NO(u-v)} : uv \in
E(P), c^tu > c^tv \}$ is the set of \Def{$c$-improving} edge directions.  If
$\NO \equiv 1$, then we also write $\ED(P) = \ED^\NO(P)$.  Note that $z \in
\ED(P)$ if and only if $-z \in \ED(P)$.

\begin{proof}[Proof of Theorem \ref{thm:NP_sweep}]
    Let $v \in V(P)$ be a vertex. It follows from the definition that the
    vertices of $\PP_{P,c}^\NO(v)$ are a subset of $\ED^\NO(P,c)$. Hence if we
    have a total order on $\ED^\NO(P,c)$ induced by a linear function $\WT$,
    then this determines the unique maximizer $\PP_{P,c}^\NO(v)^\WT$ for all
    $v$ and therefore a vertex of $\PP_{P,c}^\NO$. Since $\WT$ induces a total
    ordering on $\ED^\NO(P,c)$ if and only if $\SP(\ED^\NO(P,c))^\WT$ is a
    vertex, this proves the first claim.

    The second claim follows in the same manner.
\end{proof}

\begin{prop}\label{prop:NP_equiv}
    Let $\NO$ be a normalization with $\NO(x) > 0$ for $x \neq 0$.
    Any polytope $P$ is a weak Minkowski summand of $\NP[\GI]{P}$. Normal
    equivalence holds precisely for $2$-neighborly polytopes.
\end{prop}
\begin{proof}
    Let $c$ be a generic linear function.  It follows from convexity that $P^c
    = v$ if and only if $v$ has no $c$-improving neighbor. It follows from
    \eqref{eqn:Nhood} that $\NP[\GI]{P}(v)^c = \{0\}$. Since $\NP[\GI]{P} =
    \sum_v \NP[\GI]{P}(v)$, we see that if $(\NP[\GI]{P})^c$ is a vertex, then
    so is $P^c$. 

    A polytope $P$ is called \Def{$2$-neighborly} if any two vertices are
    adjacent.  If $P$ is $2$-neighborly, then $\NP[\GI]{P}(v) = -v + P$ and
    hence \[ \tfrac{1}{n} \NP[\GI]{P}(v) \ = \ b(P) + P \, , \] where $b(P) =
    \frac{1}{n} \sum V(P)$ is the barycenter of $P$. Now, assume that $P$ is
    not two neighborly and $u-v \not\in \NP[\GI]{P}(v)$. Then there is a
    linear function $c$ such that $\NP[\GI]{P}(v)^c$ is a vertex but $\dim P^c
    > 0$.
\end{proof}
The proof actually shows that $P$ is a weak Minkowski summand of $\NP[\NO]{P}$
for any normalization with $\NO(x) > 0$ whenever $x \neq 0$.

\begin{remark}
    It also follows from convexity that $P$ is a Minkowski summand of the edge
    zonotope $\EZ(P)$ and $\EZ(P)$ is by construction a Minkowski summand of
    $\SP(\ED(P))$. However, it is not true in general that $\EZ(P)$ is a weak
    Minkowski summand of $\NP[\GI]{P}$ nor the other way around: If $P =
    \Delta_{d-1}$, then $\EZ(P) = \Pi_{d-1}$ is a permutahedron while
    $\NP[\GI]{\Delta_{d-1}} = d \Delta_{d-1}$ up to translation. If $P$ is a
    zonotope, then $\EZ(P)$ is normally equivalent to $P$ but $\NP[\GI]{P}$
    can have more vertices than $P$.
\end{remark}

For the $d$-cube we have $\ED(C_d) = \{\pm e_1,\dots, \pm e_d\}$ and
$\SP(\ED(C_d))$ is the type-$B$ permutahedron with respect to $(1,\dots,d)$;
see~\cite[Sect.~2.2.2]{Sweep}.  In contrast to Theorem~\ref{thm:MS_zono}, it
is in general not true that the GI-neighbotope and the sweep polytope of edge directions are normally equivalent.

\begin{example}\label{ex:GINP_zono}
    Consider the zonotope $Z \subset \R^2$ for the vectors $(\pm1,1)$ and
    $(\frac{1}{2},1)$. Then $\NP[\GI]{Z}$ is a zonotope with $12$ vertices
    whereas $\SP(\ED(Z))$ has $14$ vertices.

    For the zonotope $Z$ generated by the vectors $(1,0,0), (0,1,0), (0,0,1),
    (1,1,4)$, one can check that $\NP[\GI]{Z}$ is not even a belt polytope.
\end{example}

We now show that for a very interesting class of zonotopes related to
crystallographic reflection reflection groups, normal equivalence nevertheless
holds. We refer the reader to~\cite{BjornerBrenti,Humphreys} for
more information about the geometry and combinatorics of root systems.

A finite nonempty set $\RoSy \subset \R^n \setminus \{0\}$ is a \Def{root
system} if $\RoSy \cap \R \alpha = \{-\alpha, \alpha\}$ for all $\alpha \in
\RoSy$, and $s_\alpha(\RoSy) = \RoSy$ where $s_\alpha(x) = x -
2\frac{\alpha^tx}{\alpha^t\alpha}\alpha$ is the reflection in the hyperplane
$\alpha^\perp$.  The root system is \Def{irreducible} if there is no partition
$\RoSy = \RoSy' \uplus \RoSy''$ into nonempty subsets such that $\alpha^t
\beta = 0$ for all $\alpha \in \RoSy', \beta \in \RoSy''$.  The group $W
\subset O(\R^n)$ generated by the reflections $s_\alpha$ for $\alpha \in
\RoSy$ is finite and called a \Def{reflection group}.
Define the zonotope associated to $\RoSy$
\[
    Z_\RoSy \ := \ \frac{1}{2} \sum_{\alpha \in \RoSy} [-\alpha,\alpha] \, .
\]
By construction, $Z_\RoSy$ is $W$-invariant and has edge directions
$\ED(Z_\RoSy) = \RoSy$. The sweep polytope is then
\[
    \SP(\ED(Z_\RoSy)) \ = \ \SP(\RoSy) \ = \ \sum_{\alpha,\beta \in \RoSy}
    [\alpha-\beta,\beta-\alpha] \, .
\]
The greatest improvement pivot rule is sensitive to the length of edges and in
this generality, the lengths of roots is not a meaningful invariant. A root
system is \Def{crystallographic} if $\frac{2\alpha^t\beta}{\alpha^t\alpha} \in
\Z$ for all $\alpha \in \RoSy$. Equivalently, if the group $W$ stabilizes the
lattice spanned by $\RoSy$. In this case $W$ is called a Weyl group.
Crystallographic root systems are completely classified;
see~\cite[Chapter~2]{Humphreys} and Appendix~\ref{sec:comp_roots}. In
particular, the zonotope $Z_\RoSy$ is unique up to rigid motion and homothety.

\begin{theorem}\label{thm:NP_roots}
    Let $\RoSy \subset \R^n$ be an irreducible crystallographic root system.
    Then $\NP[\GI]{Z_\RoSy}$ is normally equivalent to $\SP(\ED(Z_\RoSy))$.
\end{theorem}

Let $c \in \R^n$ be generic so that $\RoSy \cap \{ x : c^tx = 0\} = \emptyset$.
The \Def{positive system} $\Pos \subset \RoSy$ associated to $c$ is 
$\Pos := \Phi \cap \{ x : c^tx > 0\}$ and we can rewrite
\[
    Z_\Phi \ = \ \sum_{\alpha \in \Pos} [-\alpha,\alpha] \, .
\]
The sweep polytope is clearly $W$-invariant and can be rewritten as 
\[
    \SP(\ED(\RoSy)) \ = \ \sum_{\alpha,\beta \in \RoSy}
    [\alpha-\beta,\beta-\alpha]
    \ = \ 
    2 \sum_{\alpha,\beta \in \Pos} [\alpha-\beta,\beta-\alpha]
     + \sum_{\alpha,\beta \in \Pos} [-\alpha-\beta,\alpha+\beta] \, .
\]
Let $\Simp \subseteq \Pos$ be the unique minimal set of generators of the cone
$C := \cone(\Pos)$, called the \Def{simple system} of $\Pos$. Let $\vopt =
\sum_{\alpha \in \Pos}\alpha$ be the unique maximizer of $Z_\RoSy$ for the
linear function $x \mapsto c^tx$. Then
\[
    \NP[\GI]{Z_\RoSy}(\vopt) \ := \ \conv \left\{ u-\vopt : u \in
        \Nb{Z_\RoSy}{\vopt} \cup \{\vopt\} \right\} \ = \ \conv( -\Simp \cup
        \{0\} ) \, .
\]
By construction $Z_\RoSy$ is invariant under $W$ and, in fact, $W$ acts simply
transitive on the vertices. It thus follows that
\[
    \NP[\GI]{Z_\RoSy} \ = \ \sum_{w \in W} \conv( -w\Simp \cup \{0\} ) \, 
\]
and hence $\NP[\GI]{Z_\RoSy}$ is also $W$-invariant.

The dual cone $C^\vee = \{ w \in \R^n : w^t \alpha \ge 0 \text{ for all }
\alpha \in \Pos\}$ is a fundamental domain for the action of $W$ on $\R^n$.
If we want to show that $\NP[\GI]{Z_\RoSy}$ is normally equivalent to
$\SP(\ED(\RoSy))$, then it suffices to show that for all $w \in C^\vee$, if
$\SP(\ED(\RoSy))^w$ is not a vertex, then $(\NP[\GI]{Z_\RoSy})^w$ is not a
vertex. In fact, $Z_\RoSy$ is a weak Minkowski summand of both polytopes and
$Z_\RoSy^w$ is not a vertex whenever $w \in \partial C^\vee$. Thus, we 
may restrict to $w \in \int(C^\vee)$. Note that $c \in \int(C^\vee)$ and
since $\Pos$ and hence $\Simp$ are unchanged if we replace $c$ by some other
$c' \in \int(C^\vee)$, we may as well assume $w = c$.

As a first observation, note that $[-\alpha-\beta,\alpha+\beta]^c =
\alpha+\beta$ and hence
\[
    \SP(\ED(\RoSy))^c \ = \ 2\vopt + 
    2 \sum_{\alpha,\beta \in \Pos} [\alpha-\beta,\beta-\alpha]^c \, .
\]
The cone $C$ induces a partial order on $\R^n$ by $x \preceq y$ if $y-x \in
C$. If $\RoSy$ is crystallographic, then every $\alpha \in \Pos$ is a
nonnegative integer linear combination of simple roots. Hence $\alpha
\preceq \beta$ for $\alpha,\beta \in \Pos$ if and only if $\beta - \alpha =
\sum_{\gamma \in \Simp} c_\gamma \gamma$ where $c_\gamma \in \Z_{\ge 0}$. The
poset $(\Pos,\preceq)$ is called the (positive) \Def{root poset} of $\RoSy$.
Two roots $\alpha, \beta$ are \Def{incomparable} if $\alpha - \beta$ as well
as $\beta - \alpha$ are not contained in $C$. This implies that there is some
$t \in \R^n$ such that 
\[
    \SP(\ED(\RoSy))^c \ = \ t + 
    2 \sum_{\substack{\alpha,\beta \in \Pos\\ \alpha,\beta \text{
        incomparable}}} [\alpha-\beta,\beta-\alpha]^c \, .
\]
We note that if $\SP(\ED(\RoSy))^c$ is not a vertex, then there is some pair
of incomparable roots $\alpha, \beta \in \Pos$ with $c^t\alpha = c^t\beta$.

On the other hand, we observe that for $w \in W \setminus \{e\}$ and $v =
w\vopt$ the corresponding vertex of $Z_\RoSy$, we have 
\[
    \NP[\GI]{Z_\RoSy}(v)^c \ = \
    \conv( -w\Simp \cup \{0\} )^c \ = \ 
    \conv( -w\Simp \cap \Pos )^c.
\]
Indeed, if $w \neq e$, then $v \neq \vopt$. Thus $v$ has a $c$-improving edge
direction and all the improving edge directions are precisely $\Pos$. The
\Def{longest element} of $W$ is the unique $w_0 \in W$ with $w_0\Pos = -\Pos$.
In particular $ww_0\Delta = -w\Delta$ and the following result, whose proof we
give in Appendix~\ref{sec:comp_roots}, then proves Theorem~\ref{thm:NP_roots}.

\begin{theorem}\label{thm:comp_roots}
    Let $\RoSy$ be an irreducible crystallographic root system with positive
    and simple systems $\Pos \supseteq \Simp$ and let $W$ be the
    corresponding Weyl group. If $\alpha, \beta \in \Pos$ are incomparable,
    then there is $w \in W$ with $w\Simp \cap \Pos = \{\alpha,\beta\}$.
\end{theorem}

\section{Greatest-improvement and graphical neighbotopes}\label{sec:genie} 

Theorem~\ref{thm:NP_sweep} insinuates that branchings for the
greatest-improvement rule can be obtained in a \emph{greedy-like} fashion.
Indeed the corresponding arborescence is determined once the edge directions
$\ED(P)$ are sorted according to the cost vector $c$. The corresponding
neighbotopes can be viewed as solving a certain optimization problem for a
fixed polytope $P$ and varying objective function $c$. In this section, we
make the connection to greedy-like structures more precise.

Let $G = (V,E)$ be an abstract graph that throughout this section we will
assume to be simple and undirected. Let $c  \in \R^V$ be
\Def{node potentials}. For an ordered pair of adjacent nodes $(u,v)$ we
call $c_v - c_u$ the \Def{potential difference}. A \Def{branching} on $G$ is a
map $\br : V \to V$ such that $\br(v) \in \Nb{G}{v} \cup \{v\}$ and for every $v
\in V$ there is a $k \ge 0$ with $\br^k(v) = \br^{k+1}(v)$. The set $V_\br =
\{ v : \br(v) = v\}$ is the set of \Def{sinks} of the branching. The
\Def{potential energy} of a branching is
\begin{equation}\label{eqn:energy}
    c(\br) \ := \  \sum_{v \in V} c_{\br(v)} - c_v \, 
\end{equation}
and the \textsc{Max Potential Energy Branching} is the problem of finding a
branching of maximal potential energy. A scenario that we can imagine is that
$V$ is a collection of sites in a mountainous region where $c_v$ gives the
elevation. The potential difference $c_u-c_v$ is related to the energy
(coming from, say, water turbines) that can be generated by setting up a flow
from $v$ to $u$ and the edges $E$ encode the admissible connections. 
The optimization problem is now to find the energy-optimal routings from every
node to a sink. This is a particular instance of the 
Maximum Weight Branching Problem; see~\cite[Chapter~6.2]{KorteVygen}.

\newcommand\rdeg{\bar{\delta}}%
A polyhedral reformulation is apparent. Continuing Example~\ref{ex:simplex1},
let $\rdeg(\br) \in \R^V$ be the \Def{reduced in-degree sequence} of $\br$
with $\rdeg(\br)_v := |\br^{-1}(v)| - 1$. Rewriting~\eqref{eqn:energy} to
\[
    c(\br) \ = \  \sum_{v \in V} |\br^{-1}(v)| c_v - \sum_{v \in V} c_v
    \ = \  \sum_{v \in V} \rdeg(\br)_v c_v
\]
shows that we are optimizing the linear function $c$ over the
\Def{graphical neighbotope} 
\[
    \NP{G} \ = \ \conv \{ \rdeg(\br) : \br \text{ branching of } G \} \, .
\]
We call a branching $\br$ a \Def{greedy branching} if 
\[
    \br(v) \ = \ \argmax \{ c_u - c_v : u \in \Nb{G}{v} \cup \{v\} \} \, .
\]
Note that not all branchings are greedy. Indeed for vertices $v,v'$ with
$\Nb{G}{v} = \Nb{G}{v'}$ the greedy condition would force $\br(v) = \br(v')$.

\begin{theorem}\label{thm:NP_greedy}
    The vertices of $\NP{G}$ are in bijection to greedy branchings of $G$.
\end{theorem}
\begin{proof}
    As before, we note that the graphical neighbotope can be written as a
    Minkowski sum $\NP{G} = \sum_{v} \NP{G}(v)$ for
    \[
        \NP{G}(v) \ := \ \conv( e_u - e_v : u \in \Nb{G}{v} \cup \{v\} ) \, ,
    \]
    which then shows that the vertices are in one-to-one correspondence with
    greedy branchings.
\end{proof}

The greatest-improvement neighbotopes can be 
viewed as graphical neighbotopes with certain restrictions on node potentials.

\begin{prop} \label{prop:extform} 
    Let $P \subset \R^d$ a polytope with graph $G = (V,E)$. Then the
    greatest-improvement neighbotope $\NP[\GI]{P}$ is the image of 
    $\NP{G}$ under the linear projection $\pi : \R^V \to \R^d$ given by
    $\pi(e_v) := v$.
\end{prop}

So, $\NP{G(P)}$ is an \emph{extended formulation} of
$\NP[\GI]{P}$ from which an inequality description as well as the facial
structure can be recovered.

Note that greedy branchings need not be arborescences, i.e., there is not
necessarily a unique sink. For suitable node potentials, we obtain
arborescences. In particular, Proposition \ref{prop:extform} shows that any node
weighting of a polytope graph coming from applying a linear objective function
to each vertex will always yield an arborescence.

The structure underlying greedy branchings is that of a polymatroid (for details see \cite{schrijver-3book}). Recall a set
function $f_G : 2^V \to \Z_{\ge0}$ is called a \Def{polymatroid} if 
\begin{enumerate}[\rm i)]
    \item $f(\emptyset) = 0$,
    \item $f$ is non-decreasing: $A \subseteq B$ implies $f(A) \le f(B)$, and
    \item $f$ is submodular: $f(A\cup B) + f(A \cap B) \le f(A) + f(B)$,
\end{enumerate}
for all $A,B \subseteq V$. The associated \Def{polymatroid polytope} is given by 
\[
    P_f \ = \ \{ x \in \R^V_{\ge 0} : x(A) \le f(A) \text{ for
    all } A \subseteq V \} \, .
\]
where $ x(A) := \sum_{v \in A} x_v$.  The polymatroid \Def{base polytope} is
$B_f := P_f \cap \{ x : x(V) = f(V) \}$.  Polymatroids and polymatroid (base)
polytopes generalize matroids and independence polytopes. They were introduced
by Edmonds~\cite{Edmonds}, who also showed that linear functions on $P_f$ can
be maximized by a greedy-type algorithm.

\begin{prop}
    The polytope $\NP{G}$ is the polymatroid base polytope for the polymatroid
    \[
        f(S) \ = \ |S| + |\Nb{G}{S}| \, ,
    \]
    where $\Nb{G}{S} = \{ u \in V\setminus S : uv \in E \text{ for some } v
    \in S\}$. In particular, $\NP{G} = B_f$.
\end{prop}
\begin{proof}
    It follows from the description as a Minkowski sum that $\NP{G}$ is a
    generalized permutahedron in the sense of~\cite{PostPerm}. Thus the
    submodular function is given by $f(S) = \max\{ \sum_{u \in S } x_u :
    x \in \NP{G}\}$. It follows that $f(S)$ is the number of vertices $v \in
    V$ such that $S \cap (\Nb{G}{v} \cup \{v\}) \neq \emptyset$ and this is
    precisely $|S| + |\Nb{G}{S}|$.
\end{proof}

The greedy algorithm for polymatroids~\cite{Edmonds,schrijver-3book} gives us a simple
combinatorial algorithm to construct greedy branchings for given $G = (V,E)$
and generic $c \in \R^V$:
\begin{compactenum}
    \item Let $M \leftarrow \emptyset$ be the collection of marked vertices.
    \item Let $D \leftarrow \emptyset$ be the collection of already directed
        vertices.
    \item If $V = M$, STOP.  Otherwise, choose $u \in V \setminus M$ with $c_u$ maximal and $M \leftarrow M \cup \{u\}$.
    \item if $u \not\in D$, then $\br(u) := u$ and $D \leftarrow D \cup
        \{u\}$.
    \item for every $v \in \Nb{G}{u} \setminus (M \cup D)$ set $\br(v) := u$
        and $D \leftarrow D \cup \{v\}$.
    \item Repeat (3).
\end{compactenum}

The algorithm also shows that if $\br$ is a greedy branching, then there is a
vertex $u$ with $\br(v) = u$ for all $v \in \Nb{G}{u}$. This is the key to
recovering a greedy branching from its reduced indegree sequence.
\begin{compactenum}
    \item Let $M \leftarrow \emptyset$ be the collection of marked vertices.
    \item Let $D \leftarrow \emptyset$ be the collection of already directed
        vertices.
    \item If $V = M$, STOP.  Otherwise, choose $u \in V \setminus M$ with $\rdeg_u = | \Nb{G}{u} \setminus (M\cup D)|$ and
        mark $u$ ($M \leftarrow M \cup \{u\}$). 
    \item If no $u$ exists in step (3), choose any $u \in V \setminus D$, and
        $M \leftarrow M \cup \{u\}$.
    \item If no $u$ exists in steps (3) and (4), then $D = V$, and we are done. 
    \item Otherwise, for $v \in \Nb{G}{u} \setminus (M \cup D)$, direct $v$ to $u$ $D \leftarrow D \cup \{v\}$. If $u \notin D,$ direct $u$ to itself $(D \leftarrow D \cup \{u\})$. Return to Step (3). 
\end{compactenum}
Via the greedy algorithm, the $u$ with highest weight will have all of its
neighbors $\Nb{G}{u}$ directed towards it. Furthermore, the vertex of next
highest weight will have all its neighbors towards except those that have
already been directed. Hence, so long as there exist vertices that have not
been directed towards another vertex by the algorithm, there will always exist
some vertex satisfying the conditions of step (3). After no vertex satisfying
step (3) exists, all remaining vertices that are not directed must be directed
to themselves. That case is accounted for by step (4), which iterates until $D
= V$. 

The graphical neighbotopes are instances of the hypergraphic polytopes of
Benedetti et al~\cite{Benedetti}; see also~\cite{Agnarsson}. A hypergraph is a
collection of hyperedges $\mathcal{H} \subseteq 2^V$ for some finite set $V$.
The associated \Def{hypergraph polytope} is 
\[
    P_{\mathcal{H}} \ = \ \sum_{H \in \mathcal{H}} \conv\{ e_v : v \in H \}
    \, .
\]
For $G$, we can associate the hypergraph $\mathcal{H}_G = \{ \Nb{G}{v} : v \in
V \}$. It now follows from the proof of Theorem~\ref{thm:NP_greedy} that
$\NP{G} = P_{\mathcal{H}_G}$. The vertices and faces of $P_{\mathcal{H}}$ were
interpreted in terms of acyclic orientations of $\mathcal{H}$. They can be
translated directly to our greedy branchings.

\begin{example}[Complete bipartite graphs]\label{ex:NP_bipartite}
    Consider the complete bipartite graph $K_{m,n}$ with color classes $A =
    [m]$ and $B = [n]$. Let $c$ be a node potential on $K_{m,n}$ with
    corresponding greedy branching $\br$. Let $a\in A$ and $b\in B$ be the
    nodes that attain the maximal potential on $A$ and $B$, respectively. Let
    us assume that $c_a > c_b$. Then $\br(v) = a$ for all $v \in B$, $\br(u) =
    u$ for all $u \in A$ with $c_u > c_b$ and $\br(u) = b$ otherwise. Hence
    the branching is completely determined by the nodes $a,b$ and the set $S =
    \{ u \in A : c_u > c_b \}$. In particular, every such triple $(a,b,S)$ can
    occur. Exchanging the roles of $A$ and $B$ then yields the total number of
    greedy branchings as
    \[
        m \sum_{k=1}^{n} k\binom{n}{k} + n \sum_{k=1}^{m}
        k\binom{m}{k} \ = \ mn(2^{m-1} +
        2^{n-1}) \, .
    \] 
\end{example}

\begin{example}[Path graphs]\label{ex:path}
    For $n \ge 1$, let $P_n$ be the path on nodes $V = \{1,\dots,n\}$ and
    edges $E = \{\{i,i+1\} : 1 \le i < n \}$. 

    We may encode a branching $\br$ of $P_n$ uniquely as a word $W =
    W_1W_2\dots W_n$ of length $n$ over the alphabet $\{L,R,S\}$ where we set
    $W_i = L$ if $\br(i) = i-1$, $w_i = R$ if $\br(i) = i+1$ and $w_i = S$ if
    $\br(i) = i$. Note that the only forbidden subword is $RL$. This allows us
    to count all branchings. The number of branchings of $P_n$ is the number
    of walks in the following directed graph $D_1$ that start at node $1$ and
    end at node $2$ after $n$ steps: 
    \begin{center}
        \includegraphics[height=1.5cm]{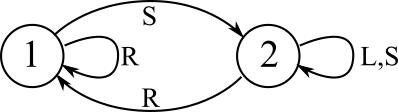}
    \end{center}
    Using the transfer matrix method~\cite[Ch.~4.7]{EC1}, one finds that the
    number of branchings is the Fibonacci number $F(2n)$.

    For the greedy branchings, one further observes that the only additional
    forbidden subword is $SS$. Hence, the number of greedy branchings of $P_n$
    is the number of walks in the following directed graph that start at node
    $1$ and end at node $2$ after $n$ steps: 
    \begin{center}
        \includegraphics[height=1.5cm]{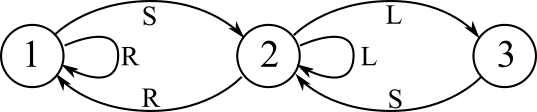}
    \end{center}
    The number of greedy branchings is given by the sequence $a(n)$ with $a(0)
    = 0$, $a(1)=1$, $a(2)=2$ and $a(n+3) = 2a(n+2) + a(n+1) - a(n)$ for $n\ge
    0$; see also~\cite{A006054}.
\end{example}

\begin{remark}
    There is an obvious graphical generalization of general neighbotopes. For
    a given graph $G = (V,E)$ let $\NO : E \to \R_{>0}$ be a normalization. We
    can define the normalized graphical neighbotope $\NP[\NO]{G}$ as the
    Minkowski sum of 
    \[
        \NP{G}(v) \ := \ \conv( \tfrac{e_u - e_v}{\NO(uv)} : u \in \Nb{G}{v} \cup \{v\} ) \, ,
    \]
    for $v \in V$. Branchings can still be found with a greedy-type algorithm
    that for every node $v$ makes the optimal choice.
\end{remark}

\appendix
\section{Proof of Theorem~\ref{thm:comp_roots}}\label{sec:comp_roots}

In this section we give a non-uniform (that is, case-by-case) proof of
Theorem~\ref{thm:comp_roots}. Let $\RoSy \subset \R^n$ an irreducible
crystallographic root system with simple and positive systems $\Simp
\subseteq \Pos = \RoSy \cap \{x : c^t x > 0 \}$.  The simple system consists
of linearly independent vectors and $r = |\Simp|$ is the rank of the root
system.  Crystallographic root systems are completely classified: There are
four infinite families $A_{n-1}, B_n, C_n, D_n$ for $n \ge 2$ as well as
sporadic instances $G_2, F_4, E_6, E_7, E_8$, where the subscript gives the
rank; see~\cite[Ch.~2.8--2.10]{Humphreys}.

For $G_2, F_4, E_6, E_7,$ and $E_8$, the number of incomparable pairs is $2,
55, 204, 546,$ and $1540$, respectively. The claim can be checked by a
computer: Let $Z_\RoSy$ be the zonotope associated to the root system.  If $v
\in V(Z_\RoSy)$ is a vertex with improving edge directions $\Simp_v$, then
$\Nb{Z_\RoSy,c}{v} = \{ s_\alpha(v) : \alpha \in \Simp_v\}$ and $u =
s_\alpha(v)$ has improving edge directions $\Simp_u = s_\alpha(\Simp')$.  This
yields a naive, yet quite fast depth-first search algorithm with starting
point $v = -\vopt$ and $\Simp_v= \Simp$. Explicit coordinates for the simple
roots are given in~\cite[Section~2.10]{Humphreys}.

In the following we check the four infinite families $A_{n-1}, B_n, C_n,$ and
$D_n$.

\subsection*{Type $A_{n-1}$}
A realization of the root system of type $A_{n-1}$ is given by $\RoSy = \{
    e_i - e_j : i,j \in [n], i \neq j\}$. For $c \in \R^n$ with $c_1 > c_2 >
\cdots > c_n$, the positive system and simple system are
\[
    \Pos \ = \  \{ e_i - e_j : 1 \le i < j \le n \} \qquad \text{ and } \qquad
    \Simp \ = \  \{ e_1 - e_2, e_2 - e_3, \dots, e_{n-1}-e_n \} \, .
\]
For $h=1,\dots,n$, let $s_h(x) = x_1 + \cdots + x_h$. The cone $C =
\cone(\Simp)$ is then given by
\[
    C \ = \ \{ x \in \R^n : s_1(x) \ge 0, s_2(x) \ge 0, \dots, s_{n-1}(x) \ge
    0 \} \, .
\]
For $i,j \in [n]$, we write $[i,j] := \{ k \in [n] : i \le k \le j \}$.

\begin{prop}[Type $A_{n-1}$ incomparable pairs]\label{prop:A-incomp}
    Let $i < j$ and $k < l$. Then
    \begin{enumerate}[\rm({A}1)]
        \item $e_i - e_j$ and $e_k - e_l$ are incomparable if and only if
            $[i,j] \not\subseteq [k,l]$ and $[k,l] \not\subseteq [i,j]$.
    \end{enumerate}
\end{prop}
There are precisely $2\binom{n}{4} + \binom{n}{3}$ pairs of incomparable
positive roots. 

\begin{proof}
    We may assume $i \le k$. If $i = k$, then $(e_k-e_l) - (e_i - e_j) = e_j -
    e_l \in C$ or $e_l - e_j \in C$. Hence $i < k$ and the first nonzero
    coordinate of $(e_k-e_l) - (e_i - e_j)$ is negative. This shows $(e_k-e_l)
    - (e_i - e_j) \not\in C$. Now, $s_h( (e_i - e_j) - (e_k - e_l)) < 0$ for
    some $h=1,\dots,n-1$ and hence $(e_i - e_j) - (e_k - e_l) \not \in C$ if
    and only if $i < j \le k < l$ or $i < k < j < l$.
\end{proof}

The reflection group $W$ associated to the $A_{n-1}$ root system acts on
$\R^n$ by permuting coordinates. For a permutation $\tau$ of $[n]$ we set
$\tau(e_i) := e_{\tau(i)}$ for $i=1,\dots,n$. In particular, $\tau\Simp = \{
    e_{\tau(1)} - e_{\tau(2)}, \dots, e_{\tau(n-1)} - e_{\tau(n)} \}$ is a
simple system and every simple system arises this way. Note that $e_{\tau(i)}
- e_{\tau(i+1)}$ is a positive root if and only if $\tau(i) < \tau(i+1)$, that
is, $i$ is an \Def{ascent} of $\tau$.

\begin{proof}[Proof of Theorem~\ref{thm:comp_roots} for $A_{n-1}$] 
    Let $\alpha = e_i - e_j$ and $\beta = e_k - e_l$ with $i < j$ and $k < l$
    be incomparable positive roots. It suffices to give a permutation $\tau$
    such that $\tau \Simp \cap \Pos = \{\alpha, \beta\}$: 
    For $i < j < k < l$
    \[
        n\, n-1\, \dots\, l+1\, l-1\, \dots\, k+1\, \mathbf{k}\,
        \mathbf{l}\, k-1\,\dots\, j+1\, \mathbf{i}\,\mathbf{j}\,j-1\,\dots\,
        i+1\,i-1\,\dots\,1 \, .
    \]
    For $i < k \le j < l$
    \[
        n\, n-1\, \dots\, l+1\, l-1\, \dots\, j+1\, \mathbf{i}\,
        \mathbf{j}\, j-1\, \dots\, k+1\, \mathbf{k}\,
        \mathbf{l}\,k-1\,\dots\,i+1\,i-1\,\dots\,1 \qedhere
    \]
\end{proof}
\subsection*{Type $B_{n}$ and $C_{n}$}
A realization of the root system of type $B_{n}$ is given by the roots $e_i -
e_j$, $e_i + e_j$ for $i,j \in [n], i \neq j$ and $\pm e_1,\dots, \pm e_n$.
For $c \in \R^n$ with $c_1 > c_2 > \cdots > c_n > 0$, the positive system and
simple system are
\[
    \Pos \ = \  \{ e_i - e_j, e_i + e_j : 1 \le i < j \le n \} \cup
    \{e_1,\dots, e_n\} \  \text{ and } \ 
    \Simp \ = \  \{ e_1 - e_2, e_2 - e_3, \dots, e_{n-1}-e_n, e_n \} \, .
\]
The cone $C = \cone(\Simp)$ is given by
\[
    C \ = \ \{ x \in \R^n : s_1(x) \ge 0, s_2(x) \ge 0, \dots, s_{n-1}(x) \ge
    0, s_{n}(x) \ge 0 \} \, .
\]
For $i,j \in [n]$, we write $(i,j) := \{ k \in [n] : i < k < j \}$.

The crystallographic root system of type $C_n$ differs from $B_n$ in that the
roots $\pm e_i$ are replaced by $\pm 2e_i$. With these modifications, the
positive system and simple system are obtained from type $B_n$. The associated
reflection group is unchanged.

\begin{prop}[Type $B_{n}$ and $C_n$ incomparable pairs] \label{prop:BC-incomp}
    Let $i < j$ and $k < l$. For type $B_n$
    \begin{enumerate}[\rm(B1)]
        \item $e_k - e_l, e_i - e_j$ are incomparable if and only if 
            $[i,j] \not\subseteq [k,l]$ and $[k,l] \not\subseteq [i,j]$;
        \item $e_k + e_l, e_i + e_j$ are incomparable if and only if  
            $[i,j] \subseteq (k,l)$ or $[k,l] \subseteq (i,j)$;

        \item $e_k - e_l, e_i + e_j$ are incomparable if and only if $k < i$;
        \item $e_k - e_l, e_i$ are incomparable if and only if $k < i$;

        \item $e_k + e_l, e_i$ are incomparable if and only if $i < k$.
    \end{enumerate}
    For type $C_n$, the cases \textrm{(B4)} and \textrm{(B5)} are replaced by 
    \begin{enumerate}[\rm(C1)]
            \setcounter{enumi}{3}
        \item $e_k - e_l, 2e_i$ are incomparable if and only if $k < i$;
        \item $e_k + e_l, 2e_i$ are incomparable if and only if $k < i < l$.
    \end{enumerate}
\end{prop}
\begin{proof} \ \\
    (B1): Since $s_n(e_s - e_t) = 0$ for all $s < t$, we have that 
$e_k - e_l, e_i - e_j$ are incomparable if and only if they are incomparable
    in type $A_{n-1}$. The claim now follows from (A1) of 
    Proposition~\ref{prop:A-incomp}.\\
    (B2): We may assume $i \le k$. If $i = k$, then $(e_k+e_l)-(e_i+e_j) = e_l
    -e_j$ or $e_j - e_l$ is in $C$. Hence $i < k$ and the first nonzero entry
    of $(e_k + e_l) - (e_i + e_j)$ is negative. Now $s_h(e_i+e_j - e_k - e_l)
    < 0$ for some $h$ if and only if $i < k < l < j$.\\
    Note that $s_n(e_k - e_l - t (e_r + e_s))  = -2t <0$ for all $r,s \in
    [n]$ and $t >0$ and hence $e_k - e_l - t (e_r +
    e_s)$ is never contained in $C$.\\
    (B3): We only need to verify $(e_i+e_j) - (e_k - e_l) \not \in C$. This is
    the case if the first nonzero coordinate is negative, which happens if and
    only if $k < i$.\\
    (B4) and (C4): Likewise, $t e_i - (e_k - e_l) \not \in C$ for $t \in
    \{1,2\}$ if and only if $k < i$.\\
    (B5): $s_n(e_i - (e_k+e_l)) = -1$ and hence $e_i - (e_k + e_l) \not \in
    C$. $(e_k + e_l) - e_i \not \in C$ if and only if $i < k$.\\
    (C5): $(e_k+e_l) - 2e_i \not \in C$ if and only if $i < l$ and 
$2e_i - (e_k+e_l) \not \in C$ if and only if $k < i$.
\end{proof}

\newcommand\bb[1]{\overline{#1}}%
The reflection group $W$ associated to the $B_n$ root system acts on $\R^n$ by
\emph{signed} permutations. A signed permutation is a pair $w = (t,\tau)$,
where $\tau$ is a permutation of $[n]$ and $t \in \{-1,+1\}^n$. Then $w$ acts
on the standard basis as $w(e_i) = t_i e_{\tau(i)}$. We represent $w$ in
\Def{window notation} and write $w = w_1\dots w_n \in
\{1,\dots,n,\bb{1},\dots,\bb{n}\}^n$ where $w_i = \bb{\tau(i)}$ if
$t_i =-1$ and $w_i = \tau(i)$ otherwise. For example, $t = (1,-1,1,-1)$ and
$\tau = (3,1,2,4)$ is denoted by $w = 3\bb{1}2\bb{3}$.
\begin{proof}[Proof of Theorem~\ref{thm:comp_roots} for $B_n$ and $C_n$] 
    Let $\alpha$ and $\beta$ be incomparable positive roots.
    For each case, we give a suitable signed permutation $w$ such
    that $w\Delta \cap \Pos = \{\alpha,\beta\}$.\\
    (B1): Let $e_i - e_j$ and $e_k - e_l$ be incomparable with $i < j$ and $i
    < k < l$. \\
    For $i < j < k < l$
    \[
        \bb{1}\, \bb{2}\, \dots\, \bb{i-1}\, \bb{i+1}\, \dots\, \bb{j-1}\,
        \bb{\mathbf{j}}\, \bb{\mathbf{i}}\, \bb{j+1}\, \dots\, \bb{k-1}\,
        \bb{k+1}\, \dots\, \bb{l-1}\, \bb{\mathbf{l}}\, \bb{\mathbf{k}}\,
        \bb{l+1}\, \dots\, \bb{n-1}\, \bb{n} \, .
    \]
    For $i < k \le j < l$
    \[
        \bb{1}\, \bb{2}\, \dots\, \bb{i-1}\, \bb{i+1}\, \dots\, \bb{k-1}\,
        \bb{\mathbf{l}}\, \bb{\mathbf{k}}\, \bb{k+1}\, \dots\, \bb{j-1}\,
        \bb{\mathbf{j}}\, \bb{\mathbf{i}}\, \bb{j+1}\, \dots\, \bb{l-1}\,
        \bb{l+1}\, \dots\, \bb{n-1}\, \bb{n}\, .
    \]

    (B2): For $e_i + e_j$ and $e_k + e_l$ with $i < k <
    l < j$
    \[
        \bb{1}\, \bb{2}\, \dots\, \bb{i-1}\, \bb{i+1}\, \dots\, \bb{k-1}\,
        \mathbf{k}\,
        \bb{\mathbf l}\, j-1\, j-2\, \dots\, l+1\, l-1\, \dots\, k+1\,
        \mathbf{i}\, \bar{\mathbf j}\,
        \bb{j+1}\, \dots\, \bb{n-1}\, \bb{n} \, .
    \]

    (B3): For $e_k - e_l$ with $k < l$ and $e_i + e_j$ with
    $k < i < j$ 
    \[
        \bb{1}\, \bb{2}\, \dots\, \bb{i-1}\, n\, n-1\, \dots\, 
        l+1\, \mathbf{k}\, \mathbf{l}\, l-1\, \dots\, i+1\, \mathbf{i}
        \, \bb{\mathbf j} \, .
    \]

    (B4) and (C4): For $e_k - e_l$ with $k < l$ and $e_i$ with $k < i$
    \[
        \bb{1}\, \bb{2}\, \dots\, \bb{i-1}\, n\, n-1\, 
        \dots\, l+1\, \mathbf{k}\, \mathbf{l}\, l-1\, \dots\, i+1\, \mathbf{i}
        \, .
    \]

    (B5) and (C5): For $e_k + e_l$ with $k < l$ and $e_i$ with $i <
    l$
    \[
        \bb{1}\, \bb{2}\, \dots\, \bb{i-1}\, \mathbf{k}\, \bb{\mathbf{l}}\,
        n\, n-1\, \dots\, l+1\, l-1\, \dots\,  i+1\,
        \mathbf{i} \, . \qedhere
    \]
\end{proof}

\subsection*{Type $D_{n}$}
A realization of the root system of type $D_n$ is given by the roots
$\pm(e_i - e_j)$ and $\pm(e_i + e_j)$ for $1 \le i < j \le n$.  For $c \in
\R^n$ with $c_1 > c_2 > \cdots > c_n > 0$, the positive and simple system
are
\[
    \Pos \ = \  \{ e_i - e_j, e_i + e_j : 1 \le i < j \le n \} 
    \quad \text{ and } \quad
    \Simp \ = \  \{ e_1 - e_2, e_2 - e_3, \dots, e_{n-1}-e_n, e_{n-1} + e_n \}
    \, .
\]
The cone $C = \cone(\Simp)$ is given by
\[
    C \ = \ \{ x \in \R^n : s_1(x) \ge 0, \dots, s_{n-1}(x) \ge 0, s_{n}(x)
    \ge 0, s_{n-1}(x) \ge x_n \} \, .
\]
\begin{prop}[Type $D_{n}$ incomparable pairs] \label{prop:D-incomp}
    Let $i < j$ and $k < l$. 
    \begin{enumerate}[\rm(D1)]
        \item $e_k - e_l, e_i - e_j$ are incomparable if and only if 
                $[i,j] \not\subseteq [k,l]$ and $[k,l] \not\subseteq [i,j]$;
        \item $e_k + e_l, e_i + e_j$ are incomparable if and only if  
            $[i,j] \subseteq (k,l)$ or $[k,l] \subseteq (i,j)$;

        \item $e_k - e_l, e_i + e_j$ are incomparable if and only if $k < i$
            or $j = l = n$.
    \end{enumerate}
\end{prop}
\begin{proof}
    The cases (D1) and (D2) follow from Proposition~\ref{prop:BC-incomp}. For
    (D3) we again note that $s_n((e_k - e_l) - (e_i + e_j)) < 0$.
    Now, $x := (e_i + e_j) - (e_k - e_l) \not \in C$ if the first nonzero
    entry is negative or $s_{n-1}(x) < x_n$. The former happens if and only if
    $k < i$. The latter is true if and only if $j=l=n$.
\end{proof}

The reflection group $W$ associated to the $D_n$ root system acts on $\R^n$ by
\emph{signed} permutations with an \emph{even} number of sign changes. Thus,
only those signed permutations $w = w_1\dots w_n \in
\{1,\dots,n,\bb{1},\dots,\bb{n}\}^n$ are permitted with an even number of
barred positions.

\begin{proof}[Proof of Theorem~\ref{thm:comp_roots} for $D_n$] 
    (D1): Let $e_i - e_j$ and $e_k - e_l$ with $i < j$ and $i < k < l$. If $n$
    is even, then the signed permutations (B1) have an even number of signs
    and should be used. Otherwise, if $n$ is odd and $l < n$, then 
    \[
        \bb{1}\, \bb{2}\, \dots\,  \bb{j-1}\,
        \bb{\mathbf{j}}\, \bb{\mathbf{i}}\, \bb{j+1}\, \dots\,  \bb{l-1}\, \bb{\mathbf{l}}\, \bb{\mathbf{k}}\,
        \bb{l+1}\, \dots\, 
        \bb{n-1} \, \mathbf{n} \, .
    \]
    If $l = n$, then  
    \[
        \bb{1}\, \bb{2}\, \dots\,  \bb{j-1}\,
        \bb{\mathbf{j}}\, \bb{\mathbf{i}}\, \bb{j+1}\, \dots\,  \bb{n-1}\,
        \bb{\mathbf{n}}\, \mathbf{k}
         \, .
    \]
    (D2): For $e_i + e_j$ and $e_k + e_l$ with $1 \le i < k <
    l < j \le n$, if $j < n$, then, depending on the parity
    \[
        \begin{aligned}
        &\bb{1}\, \bb{2}\, \dots\, 
            \bb{k-1}\,
        \mathbf{k}\,
        \bb{\mathbf l}\, j-1\, j-2\, \dots\, 
            k+1\,
        \mathbf{i}\, \bar{\mathbf j}\,
            \bb{j+1}\, \dots\, \bb{n-1}\, \bb{n}  & \text{or } \\
        &\bb{1}\, \bb{2}\, \dots\, 
            \bb{k-1}\, \mathbf{k}\,
        \bb{\mathbf l}\, j-1\, j-2\, \dots\, 
            k+1\, \mathbf{i}\, \bar{\mathbf j}\,
        \bb{j+1}\, \dots\, \bb{n-1}\, n \, .
        \end{aligned}
    \]
    If $j = n$, then, depending on the parity,
    \[
        \begin{aligned}
        &\bb{1}\, \bb{2}\, \dots\, 
        \bb{k-1}\, \mathbf{k}\,
        \bb{\mathbf l}\, j-1\, j-2\, \dots\, 
            k+1\, \mathbf{i}\, \bar{\mathbf j} & \text{or } \\
        &\bb{1}\, \bb{2}\, \dots\, 
            \bb{k-2}\,
        \mathbf{k}\,
        \bb{\mathbf l}\, j-1\, j-2\, \dots\, 
            k+1\, k-1\, \mathbf{i}\, \bar{\mathbf j}\, .
        \end{aligned}
    \]
    (D3): For $e_k - e_l$ with $k < l$ and $e_i + e_j$ with
    $k < i < j < n$ 
    \[
        \begin{aligned}
        &\bb{1}\, \bb{2}\, \dots\, \bb{i-1}\,  n-1\, \dots\, 
            k+1\, \mathbf{k}\, \mathbf{l}\, k-1\, \dots\, i+1\,
            \mathbf{i}\bb{\mathbf j}  \, n & \text{or}\\
        &\bb{1}\, \bb{2}\, \dots\, \bb{i-1}\,  n-1\, \dots\, 
            k+1\, \mathbf{k}\, \mathbf{l}\, k-1\, \dots\, i+1\,
        \mathbf{i}\bb{\mathbf j}  \, 
        \bb{n} \, .
        \end{aligned}
    \]
    If $j = l = n$ and $i = k$, then, depending on whether $n$ is odd or not:
    \[
        \bb{1}, \bb{2}, \dots, \bb{n-1},  \mathbf{i}, \bb{\mathbf{n}} 
        \quad \text{ or } \quad 
        \bb{1}, \bb{2}, \dots, \bb{n-1},  \mathbf{i}, \mathbf{n} \, .
    \]
    If $j = l = n$ and $i \neq k$, then, depending on whether $n$ is odd or not:
    \[
        \bb{1}, \bb{2}, \dots, \bb{n-1},  \mathbf{i}, \mathbf{n}, \mathbf{k}
        \quad \text{ or } \quad 
        \bb{1}, \bb{2}, \dots, \bb{n-1},  \mathbf{i}, \mathbf{n},
        \bb{\mathbf{k}} \, . \qedhere
    \]
\end{proof}

\bibliographystyle{siam} \bibliography{References,bibliolp}

\end{document}